\documentclass[11pt,reqno]{amsart}

\usepackage{marginnote,amssymb,mathrsfs,latexsym,verbatim,pdfsync,color,mathabx}

\usepackage [utf8]{inputenc}
\usepackage[colorlinks=true, linkcolor=blue]{hyperref}

\def\bbC{{\mathbb C}}
\def\bbD{{\mathbb D}}

\def\bbN{{\mathbb N}}

\def\bbR{{\mathbb R}}

\def\bbZ{{\mathbb Z}}

\def\cB{{\mathcal B}}

\def\cE{{\mathcal E}}
\def\cF{{\mathcal F}}

\def\cH{{\mathcal H}}

\def\cS{{\mathcal S}}

\def\cX{{\mathcal X}}
\def\cY{{\mathcal Y}}

\def\Re{\operatorname{Re}}

\def\la{\langle}
\def\ra{\rangle}

\def\eps{\varepsilon}

\def\vp{\varphi}
\def\ov{\overline}

\def\ms{\medskip}

\def\Hol{\operatorname{Hol}}

\def\supp{\operatorname{supp}}
\def\sinc{\operatorname{sinc}}

\def\itb{\item[{\tiny $\bullet$}]}
\def\tbi{\itb}

\def\BMO{\operatorname{BMO}}
\def\BMOA{\operatorname{BMOA}}
\def\VMO{\operatorname{VMO}}
\def\VMOA{\operatorname{VMOA}}

\def\vp{\varphi}

\def\Hol{\operatorname{Hol}}

\newtheorem{thm}{Theorem}[section]
\newtheorem{prop}[thm]{Proposition}
\newtheorem{cor}[thm]{Corollary}
\newtheorem{lem}[thm]{Lemma}
\newtheorem{defn}[thm]{Definition}

\newtheorem{remarks}[thm]{Remarks}

\newtheorem{theorem}{Theorem}

\begin{document}

\title[Duality, BMO and Hankel operators on Bernstein spaces]{Duality, BMO and Hankel operators on Bernstein spaces} 
\author[C. Bellavita, M. M. Peloso]{Carlo Bellavita, 
  Marco M. Peloso}

 \address{Dipartimento di Matematica``F. Enriques''\\
Universit\`a degli Studi di Milano\\
Via C. Saldini 50\\
I-20133 Milano}
\email{carlo.bellavita@gmail.com}

\address{Dipartimento di Matematica ``F. Enriques''\\
Universit\`a degli Studi di Milano\\
Via C. Saldini 50\\
I-20133 Milano}
\email{marco.peloso@unimi.it}

\thanks{M. Peloso is member of the Gruppo Nazionale
per l’Analisi Matematica, la Probabilità e le loro Applicazioni (GNAMPA) of
the Istituto Nazionale di Alta Matematica (INdAM). He was partially
supported by the 2022 INdAM–GNAMPA project {\em Holomorphic Functions in
One and Several Complex Variables} (CUP\_E55F22000270001).}

 \begin{abstract}
In this paper we deal with the problem of describing the dual space $(\cB^1_\kappa)^*$ of
the Bernstein space $\cB^1_\kappa$, that is the space of entire
functions of exponential type (at most) $\kappa>0$ whose restriction to the real
line is Lebesgue integrable.  We provide several characterizations,
showing that such dual space can be described as a quotient of the space
of entire functions of exponential type $\kappa$ whose restriction to
the real line are in a suitable $\BMO$-type space, or as the space of
symbols $b$ for which the Hankel opertor $H_b$ is bounded on the
Paley--Wiener space $\cB^2_{\kappa/2}$.  We also provide a
characterization of $(\cB^1_\kappa)^*$ as the $\BMO$ space w.r.t. the
Clark measures of the inner function $e^{i\kappa z}$ on the upper
half-plane, in analogy with the known description of the dual of
backward-shift invariant $1$-spaces on the torus.
Furthermore, we show that the orthogonal projection $P_\kappa:L^2(\bbR)\to \cB^2_\kappa$
induces a bounded operator form  $L^\infty(\bbR)$ onto
$(\cB^1_\kappa)^*$.

Finally, we show that $\cB^1_\kappa$ is the dual space of the suitable
$\VMO$-type space or as the space of
symbols $b$ for which the Hankel opertor $H_b$ on the
Paley--Wiener space $\cB^2_{\kappa/2}$ is compact.
\end{abstract}
\maketitle

\section{Introduction and statement of the main results}

Let $\cE_\kappa$
be the space of entire functions of exponential type at most $\kappa$, 
\begin{equation}\label{exp-type}
\cE_\kappa = \big\{ f\in\Hol(\bbC): \, \text{for every } \eps>0\
\text{there exists\ } C_\eps>0  \text{ such that \ } |f(z)|\le C_\eps
e^{(\kappa+\eps)|z|} \big\} \,.
\end{equation}
For $p\in(0,\infty]$, the Bernstein space $\cB^p_\kappa$ is defined as
$$
\cB^p_\kappa=\left\{ f\in \cE_\kappa: f_0\in L^p, \, \|f\|_{\cB^p_\kappa} = \|f_0\|_{L^p} \right\}
$$
where $f_0:=f|_{\bbR}$ denotes the restriction of $f$ to the real line
and $L^p$ is the standard Lebesgue space.
When $p=2$, the space $\cB^2_\kappa$ is the
classical Paley--Wiener space $PW_\kappa$.  We remark that the spaces $\cB^p_\kappa$
are often also called Paley--Wiener spaces and denoted as
$PW^p_\kappa$.  We select to refer them as Bernstein spaces, see the
discussion at the end of this introduction. \ms

Let 
$\cS$  and
$\cS'$ denote the space of Schwartz functions and
the space of tempered
distributions on the real line $\bbR$, resp.
For $f\in\cS'$ we
 equivalently denote by $\widehat f$ or $\cF f$ the Fourier
 transform, that for  $f\in\cS$ it is given by
 \begin{equation*}
   \widehat f(\xi)
   =  \frac{1}{\sqrt{2\pi}} \int_{\bbR} f (x)e^{-ix\xi}\, dx. 
 \end{equation*}
The Fourier transform $\cF$ is an isomorphism of $\cS$
onto itself with inverse given by
$$
\cF^{-1}g(x) =
\frac{1}{\sqrt{2\pi}}\int_{\bbR} g(\xi)e^{ix\xi}\, d\xi ,
$$
and by Plancherel Theorem, 
$\cF$ extends to a surjective
isometry $\cF: L^2(\bbR)\to L^2(\bbR)$. 
 
R. Paley and N. Wiener~\cite{PW} showed that $\cB^2_\kappa$ can be
characterized as the subspace of 
$L^2$ consisting of the functions
whose Fourier transform is supported in
$[-\kappa,\kappa]$.  There exists an analogous
characterization of the Bernstein space $\cB^p_\kappa$ as the
subspace of $L^p$ consisting of the functions 
whose Fourier transforms, as tempered distributions, are supported in
$[-\kappa,\kappa]$, see e.g.~\cite{Andersen}.
The
Paley--Wiener space $\cB^2_\kappa$ is a
reproducing kernel Hilbert space with reproducing kernel $K_w$ given by a
dilated of the normalized
sinc-function, namely
$$
K_w(z)={\textstyle \sqrt{\frac\kappa\pi} \sinc\big(\frac\kappa\pi
(z-\ov w)\big)}=: K(z,w) ,
$$
where we write $\sinc(z)=\frac{\sin(\pi z)}{\pi z}$.  The set
$\big\{\frac{1}{\sqrt{2\kappa}}e^{-i\frac\pi\kappa n(\cdot)}:\,
n\in\bbZ\big\}$ is an orthonormal basis of $L^2[-\kappa,\kappa]$, so
that
$$
\big\{ {\textstyle \sqrt{\frac\kappa\pi}}  \sinc\big( {\textstyle
  \frac\kappa\pi (\cdot-\frac\pi\kappa n)} \big) :  \,
n\in\bbZ\big\} = \big\{ K_{\frac\pi\kappa n}:\, n\in\bbZ\big\} 
$$
is an orthonormal basis of $\cB^2_\kappa$.  Notice in particular that
this is an orthonormal basis of reproducing kernels.  In particular,
for  $f\in \cB^2_\kappa$ we have the reproducing
formula\footnote{We denote by $\la\cdot\,|\,\cdot\ra$ a sesquilinear
  pairing, and by $\la\cdot,\cdot\ra$ a bilinear pairing.}
\begin{equation}\label{repr-form:eq}
f(z) = \sum_{n\in\bbZ} \la f\,|\,K_{\frac\pi\kappa n}\ra
K_{\frac\pi\kappa n}(z) = \sum_{n\in\bbZ} f \big(n{\textstyle
  \frac\pi\kappa}\big) \sinc\big( {\textstyle\frac\kappa\pi
  (z-\frac\pi\kappa n)} \big),
\end{equation}
and the renowed Whittaker--Kotelnikov--Shannon formula
\begin{equation}\label{Shannon:eq}
\| f\|_{\cB^2_\kappa}^2 = \sum_{n\in\bbZ} \big| f \big({\textstyle
  \frac\pi\kappa}n\big) \big|^2.
\end{equation}
When $p\neq2$, analouges of formulas~\eqref{repr-form:eq}
and~\eqref{Shannon:eq} hold.  Precisely, if $p\in(1,\infty)$ and
$a=(a_n)\in\ell^p(\bbZ)$, then
\begin{equation}\label{Bp-series-dev:eq}
f(z): = \sum_{n\in\bbZ} a_n {\textstyle \sqrt{\frac\kappa\pi} \sinc\big(\frac\kappa\pi z-n\big) }
\end{equation}
belongs to $\cB^p_\kappa$ and moreover $f(\frac\kappa\pi n)=a_n$, for
$n\in\bbZ$.  The convergence of the above series is in
$\cB^p_\kappa$-norm and therefore also uniformly on compact subsets of $\bbC$.
Conversely, if $f\in\cB^p_\kappa$, then
$a := \big(f\big(\frac\kappa\pi n\big)\big)_{n\in\bbZ}\in\ell^p(\bbZ)$ and
$$
f(z) = \sum_{n\in\bbZ} f(\frac\kappa\pi n) {\textstyle \sqrt{\frac\kappa\pi} \sinc\big(\frac\kappa\pi z-n\big) }.
$$
Moreover, $\|f\|_{\cB^p_\kappa} \asymp \| a\|_{\ell^p(\bbZ)}$,
\cite{P-P}.  Hence, when $p\in (1,\infty)$, the correspondence $\cB^p_\kappa\ni f\mapsto
\big(f(\frac\kappa\pi n)\big)\in\ell^p$ is an isomorphism of
$\cB^p_\kappa$ onto $\ell^p$.
We also 
refer to~\cite{PW} and~\cite[Lectures 20-21]{Levin-Lectures} for more on
the spaces $\cB^p_\kappa$, with $p\in[1,\infty]$. 

Let $P_\kappa:L^2\to\cB^2_\kappa$ be the orthogonal projection.  Then, for
$f\in L^2$, 
$$
P_\kappa f(x) = f* {\textstyle \sqrt{\frac\kappa\pi} \sinc\big(\frac\kappa\pi \cdot\big) }(x).
$$
From the boundedness of the Hilbert trasform on $L^p$,
$p\in(1,\infty)$, it follows at once that the dual of $\cB^p_\kappa$
can be identified with  $\cB^{p'}_\kappa$, under the standard (bilinear)
pairing
of duality
$$
L_g: \cB^p_\kappa \ni f \mapsto \int_\bbR f  g\, dm,
$$
for $g\in\cB^{p'}_\kappa$, $p\in(1,\infty)$,
$\frac1p+\frac{1}{p'}=1$.  It is worth pointing out that the spaces $\cB^p_\kappa$ are endowed
with an involution $\cB^p_\kappa\ni f\mapsto f^\#\in
\cB^p_\kappa$, where $f^\#(z):=\ov{f(\ov z)}$. Thus, the above pairing
of duality can be equivalently replaced by the sesquilinear pairing $\la f\,|\, g\ra =
\int_\bbR f\ov g\, dm$.  Moreover, by~\eqref{Bp-series-dev:eq}, for
$f\in \cB^p_\kappa$, $g\in \cB^{p'}_\kappa$,
$$
\int_\bbR f(x) g(x)\, dx = \sum_{n\in\bbZ} {\textstyle
  f\big(\frac\kappa\pi n\big) g\big(\frac\kappa\pi n\big)}.
$$

In~\cite{Eoff}, C.\
Eoff extended the above isomorphisms
between $\cB^p_\kappa$ and $\ell^p$ 
to the cases
$p\in(0,1]$ --  however in this discussion, we limit ourselves to the case
$p=1$.  
Let $H_{d,\alpha}$ denote the $\alpha$-translated discrete Hilbert transform\footnote{The restriction to summing on $\bbZ\setminus\{n\}$ is required only
when $\alpha=0$.}
\begin{equation}\label{Hd:def}
H_{d,\alpha}a (n) =\sum_{k\neq n} \frac{a_k}{n-k+\alpha},
\end{equation}
where $a=(a_n) \in \ell^1$. Then,
we
define
$H^1=H^1(\bbZ)$ the Hardy space on $\bbZ$ as
$$
H^1 = \big\{a=(a_n)_{n\in\bbZ}:\, a,\, H_{d,\alpha}a \in \ell^1 \big\},
$$
for one (equivalently, all) $\alpha\in[0,1)$, 
with norm $\|a\|_{H^1} = \|a\|_{\ell^1}+ \|H_{d,\alpha} a\|_{\ell^1}$.
Eoff proved that if $f\in \cB^1_\pi$ then $a :=
\big((-1)^nf\big(\frac\kappa\pi n)\big)_{n\in\bbZ}\in H^1$, and
conversely, it $a=(a_n)\in H^1$, there exists a unique $f\in \cB^1_\pi$
such that $(-1)^n f\big(\frac\kappa\pi n\big) =a_n$, for all
$n\in\bbZ$. Moreover, $\|f\|_{\cB^1_\pi}\asymp 
\|\big((-1)^nf\big(\frac\kappa\pi\cdot )\big) \|_{H^1}$.
The dual of $H^1(\bbZ)$ can be identified with $\BMO(\bbZ)$ 
with the pairing of duality
$$
\la h,b\ra := \sum_{n\in\bbZ} h_n b_n,
$$
where $h=(h_n)\in H^1(\bbZ)$, $b=(b_b)\in \BMO(\bbZ)$.  The above
pairing of duality must be interpreted in a weak sense, since the sum may
not converge absolutely, in general, using the atomic characterization 
of $H^1(\bbZ)$ proved in ~\cite{Boza-Carro}.
Moreover, the $\BMO(\bbZ)$-norm  is equivalent to the norm of duality
with $H^1(\bbZ)$.

Given the results cited above, it is fairly straightforward to obtain
a description of the dual of $\cB^1_\kappa$ as a space of sequences.
Precisely, given any $L\in (\cB^1_\kappa)^*$, there exists a unique
$b\in \BMO(\bbZ)$ such that
$$
L(f) = \la \big((-1)^nf(n)\big),b\ra,
$$
with $\|L\|_{(\cB^1_\kappa)^*} \asymp \|b\|_{\BMO(\bbZ)}$, and
conversely any $b\in \BMO(\bbZ)$ defines a continuous linear functional
on $H^1(\bbZ)$.  Such a description of the dual space
$(\cB^1_\kappa)^*$ is analogous with the characterization of the dual
space of coinvariant subspaces of $H^1(\bbD)$~\cite{Bessonov}, where
$\bbD$ denotes the unit disk, and $H^1(\bbD)$ is the classical Hardy
space on $\bbD$; see also~\cite{student-Partington}.
\ms

The main goal of this work is to describe the dual of $\cB^1_\kappa$
as a space of entire functions, extending the pairing of duality
between $\cB^p_\kappa$  and $\cB^{p'}_\kappa$, $p\in(1,\infty)$, $\frac1p+\frac{1}{p'}=1$.
One of  characterizations that we obtain involves, in a rather natural
and classical way, the symbols of bounded Hankel operators, that we
now describe.  Given a locally integrable function $\vp$ on $\bbR$,
let $H_\vp$ denote the Hankel operator on $\cB^2_\kappa$   (that is, the
Paley--Wiener space $PW_\kappa$) with symbol $\vp$, given by, 
\begin{equation}\label{Hankel-op:eq}
  H_\vp(f)= P_\kappa (\vp\ov f) : \cB^2_\kappa\to \cB^2_\kappa,
\end{equation}
initially defined for $f$ in a dense subspace of $\cB^2_\kappa$. Such operators where first studied by R.\
Rochberg~\cite{Rochberg} who obtained a characterization of bounded
and compact operators, in terms of properties of their symbols.  We
refer to  Section~\ref{Pf-Thm2:sec} for a discussion of these results.
\ms

Before stating our main results, we need a few more definitions and remarks.
Since, for $\lambda>0$, the mapping $\cB^p_\kappa \ni f\mapsto
\lambda^{-1/p} f(\lambda\cdot)\in \cB^p_{\lambda\kappa}$ is an
isometric isomorphism, without loss of generality, we may take
$\kappa=\pi$.
Since there is no loss of generality in considering
only the case $\kappa=\pi$, we are going to do so throught the
paper.

Let $\BMO(\bbR)$ denote the space of functions (modulo
constants) of bounded mean oscillation:
$$
\BMO(\bbR) = \Big\{ \vp \in L^1_{\rm loc}(\bbR):\, \|\vp\|_{\BMO(\bbR)}:= \sup_{I\subset\bbR}
\frac{1}{|I|} \int_I |\vp -\vp_I|\, dm <\infty \
\Big\},
$$
where $I\subset\bbR$ is any bounded interval,
$\vp_I=\frac{1}{|I|}\int_I \vp$ is the average of $\vp$ over
$I$, and $|E|$ denotes the Lebesgue measure of a measurable set
$E$.\footnote{Adopting the same notation for $\BMO(\bbR)$ and
  $\BMO(\bbZ)$ should cause no confusion.}
We also define the space of analytic $\BMO$,
$$
\BMOA = \big\{ \vp\in \BMO(\bbR):\, \supp (\widehat
\vp)\subseteq(0,\infty)\big\}.
$$
Notice that $\ov{\BMOA}=\big\{ \vp\in \BMO(\bbR):\, \supp (\widehat
\vp)\subseteq(-\infty,0)\big\}$.

We define a space of equivalence classes of functions in $\cE_\pi$.
\begin{defn}\label{Yspace:def}{\rm
Let
$$
    \BMO(e^{-i\pi z}) = \Big\{ f\in \cE_\pi:
      \, e^{i\pi(\cdot)}f_0
   \in \BMOA,\,  \, e^{-i\pi(\cdot)}f_0
   \in \ov{\BMOA} \Big\}\Big/\operatorname{span}\{e^{\pm
     i\pi(\cdot)}\Big\} ,
   $$
    with norm
\begin{equation}\label{Y-norm:def}
      \|f\|_{\BMO(e^{-i\pi z})}=  \inf_{c\in\bbC} \| e^{i\pi(\cdot)}f_0 + ce^{i2\pi(\cdot)}\|_{\BMOA}
 + \inf_{c\in\bbC} \| e^{-i\pi(\cdot)}f_0 + ce^{-i2\pi(\cdot)}\|_{\ov{\BMOA}}    .
\end{equation}
  }
\end{defn}  

\begin{remarks}{\rm \ 
    
(1) Observe that the definition of $\BMO(e^{-i\pi z})$ makes sense, since the
functions $g_\pm:=e^{\pm i\pi z}\in\cE_\pi$ and that $e^{\pm i\pi
  z}g_\pm \in L^\infty \hookrightarrow\BMO(\bbR)$.

(2) Notice that if $f\in\BMO(e^{-i\pi z})$, then $f_0\in\cS'$, and by the well-known properties
of $\BMO(\bbR)$, $\int_\bbR|f(x)|\, \frac{dx}{1+x^2}<\infty$.

(3) With the given norm, $\BMO(e^{-i\pi z})$ is a Banach space, see Theorem ~\ref{Y-Banach:prop}.
} \ms
\end{remarks}

Our first result is the identification of $(\cB^1_\pi)^*$ with the
spaces $\BMO(e^{-i\pi z})$.  Recall that,
$$
(\cB^1_\pi)_{|_\bbR} =\big\{ g\in L^1(\bbR): \supp(\widehat g)\subseteq[-\pi,\pi]\big\},
$$
see e.g.~\cite{Andersen,Bernstein-CR}. This implies that, since if
$f\in\cB^1_\pi$, $\widehat{f_0}(\pm\pi)=0$, 
the bounded functions $e^{\pm i\pi(\cdot)}$ annihilate
$\cB^1_\pi$. Moreover, we point out the subspace $\{f\in \cB^1_\pi:\, f_0\in\cS\}$
is dense in $\cB^1_\pi$, see Lemma~\ref{density-lem}.

\begin{theorem}\label{main-thm1bis}
{\rm  (1)} The projection $P_\pi$ induces a surjective, bounded linear operator
  $P_\pi :L^\infty\to\BMO(e^{-i\pi z})$ such that
  $\|f\|_{\BMO(e^{-i\pi z})}\approx \inf\{ \|\vp\|_{L^\infty}:\, P_\pi \vp=f\}$.

{\rm  (2)} The space $\BMO(e^{-i\pi z})$ can be identified with $(\cB^1_\pi)^*$. Precisely,
each $f\in \BMO(e^{-i\pi z})$ defines a bounded linear functional on $\cB^1_\pi$ by
setting
$$
L_f(h)= \int_\bbR h(x) f(x)\, dx
$$
for any $h\in \cB^1_\pi$, 
and conversely, 
  for each $L\in (\cB^1_\pi)^*$ there exists a unique
$f\in \BMO(e^{-i\pi z})$ such that $L=L_f$.  Moreover, 
$\|L_f\|_{(\cB^1_\pi)^*}\approx\|f\|_{\BMO(e^{-i\pi z})}$.  
\end{theorem}

The above pairings of duality are not absolutely convergent in general
and must be interpreted in a weak sense.  Given $h\in
\cB^1_\pi$, let $h_{(\delta)}\in  \cS_\pi $ and $h_{(\delta)}\to h$ in $ \cB^1_\pi$ as
$\delta\to0^+$,  
then
\begin{equation}\label{pairing-duality}
\int_\bbR h(x) f(x)\, dx := \lim_{\delta\to0^+} \int_\bbR
h_{(\delta)}(x)f(x)\, dx;
\end{equation}
the existence of such functions $h_{(\delta)}$ will be guaranteed by Lemma~\ref{density-lem}.
 \ms

We now characterize the dual space $(\cB^1_\pi)^*$ as the image of
$L^\infty(\bbR)$ via the orthogonal projection $P_\pi$ and  in terms
of bounded Hankel operators on $\cB^2_{\pi/2}$.

\begin{theorem}\label{main-thm2}
  Let $\vp\in
  \cS'$ be a locally integrable function.  Then, the following
conditions are equivalent:
\begin{itemize}
\item[(1)] there exists $\widetilde \vp\in L^\infty$ such that $P_\pi(\widetilde \vp)=\vp$;
\item[(2)] 
  $\vp$ defines an element of $(\cB^1_\pi)^*$ with the pairing
  of duality as in~\eqref{pairing-duality};
\item[(3)] the Hankel operator with symbol $\vp$, $H_\vp: \cB^2_{\pi/2}\to
  \cB^2_{\pi/2}$ is bounded.
\end{itemize}
Moreover, the quantities
$$
\inf_{\widetilde \vp\in L^\infty:\, P_\pi(\widetilde \vp)=\vp }\|\widetilde
\vp\|_{L^\infty},\quad \| \vp\|_{(\cB^1_\pi)^*},\quad \|H_\vp\|_{ \cB^2_{\pi/2}\to
  \cB^2_{\pi/2}}
$$
are all comparable.
\end{theorem}

\ms

As metioned earlier, it has been {\em folklore} for sometime that the
dual space $(\cB_\pi^1)^*$ can be described as a space of sequences.
We now make this description explicit.

\begin{defn}{\rm
 Let $\alpha\in[0,1)$. We define
$$
    \cX_\alpha  = \Big\{ f\in \cE_\pi:\,    f(iy) =o(|y|e^{\pi|y|})\ 
      \text{for}\ |y|\to\infty, \, \big(e^{i\pi(\cdot+\alpha)}f
    (\cdot+\alpha)\big)_{|_{\bbR}}\in \BMO(\bbZ) \Big\}\Big/\operatorname{span}\{e^{\pm
        i\pi(\cdot)}\} ,
      $$
      with norm
      $$\|f\|_{\cX_\alpha}= \inf_{c_\pm\in\bbC} \big\|
      e^{i\pi(\cdot+\alpha)}\big( f(\cdot+\alpha)
      +c_+ e^{i\pi(\cdot)}+c_- e^{-i\pi(\cdot)}\big) \big\|_{\BMO(\bbZ)} =\big\|
      e^{i\pi(\cdot+\alpha)}\big( f(\cdot+\alpha)\big\|_{\BMO(\bbZ)} .
      $$
  }
\end{defn}

\begin{remarks}{\rm \ 
    
(1) Observe that the definition of $\cX_\alpha$ makes sense, since the
functions $e^{\pm i\pi z}$ are in $\cE_\pi$ and satisfy the two other
conditions in the definition of $\cX_\alpha$, namely that $ f(iy)
=o(|y|e^{\pi|y|})$
for $|y|\to\infty$, and $\big(e^{i\pi(\cdot+\alpha)}f
    (\cdot+\alpha)\big)_{|_{\bbZ}}\in \BMO(\bbZ)$.  Moreover, notice
    that $\cX_\alpha\neq\{0\}$  since it contains $\sinc(\cdot-\alpha)$.

(2) With the given norm, $\cX_\alpha$ is a Banach space, see Proposition~\ref{Xalpha-Banach:prop}.

(3) We point out  that
the spaces $\cX_\alpha$ are related to the $\BMO$ spaces with respect
to the
the Clark measures  of the inner function $e^{i2\pi z}$, see the
remark in Section~\ref{final-sec}.
} \ms
\end{remarks}

\begin{defn}\label{Talpha:defn}{\rm 
    For $\alpha\in[0,1)$, let $T_\alpha$ be  the linear  mapping 
initially defined on sequences of compact support,
given by
\begin{equation}\label{Talpha:eqq}
  T_\alpha a(z) = e^{-i\pi\alpha}\sum_{n\in\bbZ\setminus\{0\}} (-1)^n a_n \Big(
  \frac{1}{z-\alpha-n} +\frac{1}{n}\Big)
  \frac{\sin[\pi(z-\alpha-n)]}{\pi} + e^{-i\pi\alpha} a_0 \sinc(z-\alpha).
\end{equation}
}
\end{defn}

For the spaces $\cX_\alpha$ we prove the following results.

\begin{theorem}\label{main-thm1}
For $\alpha\in[0,1)$,  $T_\alpha$ induces an isomorphism of Banach spaces
  $T_\alpha:\BMO(\bbZ)\to
\cX_\alpha$ (in particular, it is
onto).  Moreover, for each $L\in (\cB^1_\pi)^*$ there exists a unique
$f\in \cX_\alpha$ that represents the functional $L$ in the sense that
for any $h\in \cB^1_\pi$, 
\begin{equation}\label{discrete-pairing-duality}
L(f) = \sum_{n\in\bbZ} h(n+\alpha)f(n+\alpha).
\end{equation}
\end{theorem}

It turns out that the spaces $\cX_\alpha$, $\alpha\in[0,1)$, and $\BMO(e^{-i\pi z})$
they all coincide, thus providing a discrete characterization of the
dual of $\cB^1_\pi$.
\begin{theorem}\label{Y=Xalpha:thm}
For each $\alpha\in[0,1)$, $\BMO(e^{-i\pi z})=\cX_\alpha$, with equivalence of norms.
\end{theorem}

\ms

Our next result shows that $\cB^1_\pi$ is the dual space of a
$\VMO$-type space that can also be described in terms of compact Hankel
operators.  We recall that $\VMO$ is the subspace of $\BMO$ consisting
of the closure of $C_0(\bbR)$ in the $\BMO$-norm -- to be precise, the 
closure of $C_0(\bbR)$ {\em modulo constants}.  $\VMO$ can be
charracterized as the subspace of $\BMO$-functions $\vp$ such that
$$
\lim_{\delta\to0} \sup_{|I|<\delta} \frac{1}{|I|} \int_I |\vp(x)-
\vp_I|\, dx =0.
$$
We also set $\VMOA:=\VMO\cap\BMOA$.

\begin{defn}\label{Y0space:def}{\rm
We define $\VMO(e^{-i\pi z})$ to be the subspace of $\BMO(e^{-i\pi z})$ given by
$$
    \VMO(e^{-i\pi z}) = \Big\{ f\in \cE_\pi:
      \, e^{i\pi(\cdot)}f_0
   \in \VMOA,\,  \, e^{-i\pi(\cdot)}f_0
   \in \ov{\VMOA} \Big\}\Big/\operatorname{span}\{e^{\pm
        i\pi(\cdot)}\} ,
      $$
      with the norm of $\BMO(e^{-i\pi z})$. 
} \ms
\end{defn}

\begin{theorem}\label{main-thmVMO}
{\rm  (1)} The projection $P_\pi$ maps $C_0(\bbR)$ onto $\VMO(e^{-i\pi z})$ and for all $f\in\VMO(e^{-i\pi z})$,
  $\|f\|_{\BMO(e^{-i\pi z})}\approx \inf\{ \|\eta\|_{L^\infty}:\, \eta\in C_0(\bbR),\,  P_\pi \eta=f\}$.

 {\rm (2)} The space $\cB^1_\pi$ can be identified with ${\VMO(e^{-i\pi z})}^*$,  the dual of $\VMO(e^{-i\pi z})$. Precisely,
each  $h\in\cB^1_\pi$ defines a bounded linear functional on $\VMO(e^{-i\pi z})$ by
the pairing
$$
L_h(\vp)= \int_\bbR h(x) \vp(x)\, dx =\int_\bbR h(x)\eta(x)\, dx
$$
for any $\vp\in \VMO(e^{-i\pi z})$, where $\vp=P_\pi\eta$, $\eta\in C_0(\bbR)$. 
Conversely, 
  for each $L\in {\VMO(e^{-i\pi z})}^*$ there exists a unique
$h\in\cB^1_\pi$ such that $L=L_h$.  Moreover, 
$\|L_h\|_{{\VMO(e^{-i\pi z})}^*}\approx\|h\|_{\cB^1_\pi}$.  
\end{theorem}

\begin{theorem}\label{main-thm3}
Let $\vp\in\cE_\pi$ be such that $\vp_0\in\cS'$.  Then, the following
conditions are equivalent:
\begin{itemize}
\item[(1)]  $\vp\in \VMO(e^{-i\pi z})$;
\item[(2)] the Hankel operator with symbol $\vp$, $H_\vp: \cB^2_{\pi/2}\to
  \cB^2_{\pi/2}$ is compact.
\end{itemize}

\end{theorem}

\ms

The Paley--Wiener and Bernstein spaces constitute a cornerstone of
function theory, analysis of holomorphic function spaces, signal
analysis and operator theory.  We point out that the spaces
$\cB^p_\kappa$ with $p\neq2$ are often also called ``Paley--Wiener spaces''
in the literature.  The classical Bernstein inequality
\cite{Bernstein}
(see also~\cite{ P-P, Andersen, Bernstein-CR}) was first proved in the case
of $\cB^\infty_\kappa$. The extension
to $\cB^p_\kappa$ for $p<\infty$ (and $p\neq2$) is simple.  For this
reason, the spaces $\cB^p_\kappa$ are often called {\em Bernstein}
spaces. In this work , we adhere to such terminology.

While in this paper we deal only with
their $1$-dimensional version, the higher dimensional generalizations
have drawn a considerable amount of interest in recent times.  We
particularly refer to the papers
\cite{fractional-Bernstein,sampling-Cn+1,Bernstein-CR,Carleson-on-Bernstein},
and references therein. 

The paper is organized as follows.  In Section~\ref{Pf-Thm1bis:sec} we
prove some properties of the space $\BMO(e^{-i\pi z})$ and 
that the dual space $(\cB^1_\pi)^*$ can be identified with $\BMO(e^{-i\pi z})$, and
thus prove Theorem~\ref{main-thm1bis}.  In Section~\ref{Pf-Thm2:sec}
we prove the characterization of $\BMO(e^{-i\pi z})$ in terms of boundedness of
Hankel operators on $\cB^2_{\pi/2}$, in particular we prove
Theorem~\ref{main-thm2}. Section~\ref{Pf-Thm1:sec} is devoted to the
discrete characterization of the pairing of duality, along the lines
of~\cite{Bessonov}. In Section~\ref{VMO:sec} we show that $\cB^1_\pi$
is the dual space of a VMO-type space, whose elements are in turn
precisely the symbols of compact Hankel operators on $\cB^2_{\pi/2}$.
We conclude with some remarks about extensions, open problems, and future works.


\ms

\section{Proof of Theorem~\ref{main-thm1bis}}\label{Pf-Thm1bis:sec}

We begin by collecting a few preliminary results. The first result can be proved exactly as~\cite[Lemma
4.1]{fractional-Bernstein}. That was stated for $p\in(1,\infty)$, but
the proof applies {\em verbatim} also to the case $p=1$. 
\begin{lem}\label{density-lem}
The subspace $\cS_\pi:= \{f\in \cB^1_\pi:\, f_0\in\cS\}$ is dense
in $\cB^1_\pi$.
\end{lem}

We are going to need the following, most likely known, simple lemma.
\begin{lem}\label{exp-growth:lem}
Let $F\in\cE_\pi$ be such that $F(n)=0$ for all $n\in\bbZ$, and that it
satisfies condition $(Y)$: $F(iy)=o\big(|y|e^{\pi|y|}\big)$. Then, there
exists $a\in\bbC$ such that $F(z)=a\sin(\pi z)$. 
\end{lem}

\proof
Let $F$ be as in the statement. There exists an entire function $G$ of
order at most $1$
such that $F(z)=\sin(\pi z)G(z)$. Denoting by $I_f$ the indicator
diagram of an entire function with respect to the order $1$ we have,
$I_F = [-i\pi,i\pi]+ I_G$. By the properties of indicator diagrams
(see~\cite[Lectures 8-9]{Levin-Lectures}) we obtain that $G$ is of exponential
type $0$, that is, $G\in\cE_0$.  Let us consider $G(z)/(z+1)$ on
the right half-plane and $G(z)/(z-1)$ on the left half-plane.  By
applying the Phragm\'en--Lindel\"of principle to these functions we
see that $G$ is a polynomial of degree at most $1$, $G(z)=a+bz$.
Therefore,
$F(z)=a\sin(\pi z) +bz\sin(\pi z)$.  But, if $b\neq0$, $F$ does not
satisfy condition $(Y)$, whence the conclusion.
\qed

\begin{thm}\label{Y-Banach:prop}
With the given norm, $\BMO(e^{-i\pi z})$  is a Banach space. 
   \end{thm}

   \proof 
We have observed that $\BMO(e^{-i\pi z})$ is well defined.  In order
  to prove that the given one is a norm, we only need to prove the
  triangular inequality and that if $\|f\|_{\BMO(e^{-i\pi z})}=0$, then
  $f=0$.
Recall that
\begin{align*}
    \|f\|_{\BMO(e^{-i\pi z})} & =  \inf_{c_+\in\bbC} \| e^{i\pi(\cdot)}f_0 + c_+e^{i2\pi(\cdot)}\|_{\BMOA}
 + \inf_{c_-\in\bbC} \| e^{-i\pi(\cdot)}f_0 + c_-e^{-i2\pi(\cdot)}\|_{\ov{\BMOA}}    .
\end{align*}
We claim that each of the two infima is actually a
minimum.  It suffices to consider the first one.
Let $m$ be such an infimum, and
  $(c_k)$ be such that $ \| e^{i\pi(\cdot)}f_0 + c_k e^{i2\pi(\cdot)}\|_{\BMOA}\to m$ as $k\to\infty$. Then,
  $(c_k)$ must be bounded, since otherwise
$\| e^{i\pi(\cdot)}f_0 + c_k e^{i2\pi(\cdot)}\|_{\BMOA}\ge \| c_k e^{i2\pi(\cdot)}\|_{\BMOA}- \| e^{i\pi(\cdot)}f_0 \|_{\BMOA}$, which tends to $\infty$ as
      $|c_k|\to \infty$.
Then, we can assume $c_k\to c$ as $k\to\infty$. Then,  the sequence of
functions in
$\BMO(\bbR)$, $\big(  e^{i\pi(\cdot)}f_0 + c_ke^{i2\pi(\cdot)}\big)_{k\in\bbN}$ is a Cauchy sequence in $\BMO(\bbR)$, since
$$
  \big\|  e^{i\pi(\cdot)}f_0 + c_ke^{i2\pi(\cdot)} -
  \big( e^{i\pi(\cdot)}f_0 + c_j e^{i2\pi(\cdot)}\big) \|_{\BMO(\bbR)} 
= |c_k-c_j| \|e^{i2\pi(\cdot)}\|_{\BMO(\bbR)} \le 2|c_k-c_j|.
$$
This shows that each of the infima is a minimum.
Therefore, there exist $c_{f,\pm}\in \bbC$ such that 
\begin{align}\label{Y-norm-eq}
\| f\|_{\BMO(e^{-i\pi z})} & = \| e^{i\pi(\cdot)}f_0 + c_{f,+}e^{i2\pi(\cdot)}\|_{\BMOA}
             + \| e^{-i\pi(\cdot)}f_0 + c_{f,-}
e^{-i2\pi(\cdot)}\|_{\ov{\BMOA}} .
  \end{align}

  Suppose that $\| f\|_{\BMO(e^{-i\pi z})} =0$. Then, 
$f_0 =-c_{f,\pm} e^{\pm i\pi(\cdot)}$ in $\BMO(\bbR)$, so
that
$$
f = c_+ e^{i\pi(\cdot)} + c_- e^{-i\pi(\cdot)},
$$
i.e. $f=0$ in $\BMO(e^{-i\pi z})$. We observe in passing that the non-zero constans
give rise to a non-trivial linear funcionals on $\cB^1_\pi$, since
$\widehat f(0)\neq0$ in general.

Now, given $f,g\in\BMO(e^{-i\pi z})$, by~\eqref{Y-norm-eq}  we have that
\begin{align*}
\|f+g\|_{\BMO(e^{-i\pi z})} & \le 
\| e^{i\pi(\cdot)}(f_0 +g_0)+ (c_{f,+}+c_{g,+}) e^{i2\pi(\cdot)}\|_{\BMOA}\\
&\qquad\qquad + \| e^{-i\pi(\cdot)}(f_0 +g_0) + (c_{f,-}+c_{g,-})e^{-i2\pi(\cdot)}\|_{\ov{\BMOA}}  \\
& \le \| f\|_{\BMO(e^{-i\pi z})} +\| g\|_{\BMO(e^{-i\pi z})} .
\end{align*}
Finally, we check that $\BMO(e^{-i\pi z})$ is complete in its norm. Let
$\{f_{(n)}\}$ be a Cauchy sequence in $\BMO(e^{-i\pi z})$. By the definition of the
norm in $\BMO(e^{-i\pi z})$, see \eqref{Y-norm:def}, it then follows that 
$\{ e^{i\pi(\cdot)} f_{(n)}\}$ is a Cauchy sequence in
$\BMOA/\operatorname{span}\{e^{i2\pi(\cdot)}\}$, and 
$\{ e^{-i\pi(\cdot)} f_{(n)}\}$ is a Cauchy sequence in
$\ov{\BMOA}/\operatorname{span}\{e^{-i2\pi(\cdot)}\}$.
Such spaces are complete since, for instance,  $\BMOA/\operatorname{span}\{e^{i2\pi(\cdot)}\}$ is
 is the dual space of $\{f\in H^1(\bbC^+):\,
\widehat f(2\pi)=0\}$.  Therefore, there exist 
$F\in\BMOA/\operatorname{span}\{e^{i2\pi(\cdot)}\}$ and $G\in
\ov{\BMOA}/\operatorname{span}\{e^{-i2\pi(\cdot)}\}$
such that
$$
 e^{i\pi(\cdot)} f_{(n)}\to F \text{ in }
 \BMOA/\operatorname{span}\{e^{i2\pi(\cdot)}\},\qquad
 e^{-i\pi(\cdot)} f_{(n)} \to G \text{ in } \ov{\BMOA}/\operatorname{span}\{e^{-i2\pi(\cdot)}\}
$$
as $n\to\infty$. It follows that,
$$
e^{i\pi(\cdot)} f_{(n)}\to F \text{ in }
\cS'/\operatorname{span}\{e^{i2\pi(\cdot)}, 1\},
\qquad
 e^{-i\pi(\cdot)} f_{(n)} \to G \text{ in } \cS'/\operatorname{span}\{e^{-i2\pi(\cdot)},1\}.
$$
Notice again that $\cS'/\operatorname{span}\{e^{i2\pi(\cdot)}, 1\}$ is
the dual of $\{\psi\in\cS:\, \widehat\psi(2\pi)=\widehat\psi(0)=0\}$,
and similarly for $\cS'/\operatorname{span}\{e^{-i2\pi(\cdot)},1\}$.
The convergences  imply that
$$
f_{(n)}\to e^{-i\pi(\cdot)}F\text{ in }
\cS'/\operatorname{span}\{e^{i\pi(\cdot)}, e^{-i\pi(\cdot)}\},
\qquad
f_{(n)}\to e^{i\pi(\cdot)}G\text{ in }
\cS'/\operatorname{span}\{e^{i\pi(\cdot)}, e^{-i\pi(\cdot)}\}.
$$
We conclude that
\begin{equation}\label{f-tilde:def}
e^{-i\pi(\cdot)}F=
e^{i\pi(\cdot)}G=:\widetilde f\in\cS'/\operatorname{span}\{e^{i\pi(\cdot)},
e^{-i\pi(\cdot)}\}.
\end{equation}
Having constructed the limit $f$ as an element of
$\cS'/\operatorname{span}\{e^{\pm i\pi(\cdot)}\}$, we want to show that 
$\widetilde  f$
(in fact, any representative in its
equivalence class) is
the restriction of (an) $f\in\cE_\pi$, that is, $\widetilde f=f_0$,
where $f\in\BMO(e^{-i\pi z})$.   Since $F\in \BMOA$ and $G\in\ov{\BMOA}$,
$$
\supp (\widehat{e^{-i\pi(\cdot)}F})\subseteq[-\pi,\infty), \qquad
\supp (\widehat{e^{i\pi(\cdot)}G})\subseteq(-\infty,\pi],
$$
we have that
$$
\supp (\widehat{\widetilde f})\subseteq[-\pi,\pi].
$$
The classical Paley--Wiener--Schwartz theorem (see e.g.~\cite[Theorem
2, Chapter I]{Triebel}), $\widetilde f$ is indeed the restriction to
$\bbR$ of a (unique) function $f\in\cE_\pi$. By the definition of
$\widetilde f$ \eqref{f-tilde:def} it is clear that $f\in\BMO(e^{-i\pi z})$.
Finally, 
\begin{align*}
\|f_{(n)}-f\|_{\BMO(e^{-i\pi z})}
& = \inf_{c,d\in\bbC} \| e^{i\pi(\cdot)}(f_{(n)})_0- F + ce^{i2\pi(\cdot)}\|_{\BMOA}
 +  \| e^{-i\pi(\cdot)}(f_{(n)})_0- G + de^{-i2\pi(\cdot)}\|_{\ov{\BMOA}} \\
  & = \| e^{i\pi(\cdot)}(f_{(n)})_0- F \|_{\BMOA/\operatorname{span}\{e^{i2\pi(\cdot)}\}}
    +\| e^{-i\pi(\cdot)}(f_{(n)})_0- G\|_{\ov{\BMOA} /\operatorname{span}\{e^{-i2\pi(\cdot)}\} } \\
  & \to 0,
\end{align*}
as $n\to\infty$. This completes the proof. \qed
\ms

\proof[Proof of Theorem~\ref{main-thm1bis}]
(1) We show that $P_\pi:L^\infty \to \BMO(e^{-i\pi z})$ is bounded.  We adopt the
following notation: Given an interval $I$, we denote by
$P_I$ the orthogonal projection of $L^2(\bbR)$ onto the subspace
of functions with support in $I$.  If $I=[-\kappa,\kappa]$, we
keep writing $P_\kappa$ in place of  $P_{[-\kappa,\kappa]}$.  If
$I=[0,\infty)$  ($I=(-\infty,0]$, resp.), we write $P_+$ ($P_-$,
resp.).  We recall that $P_\pm:L^\infty\to\BMOA^\pm$, where, for later
convenience, we write $ \BMOA^+=\BMOA$, and $\BMOA^-=\ov{\BMOA}$, are
bounded and surjective, see e.g.~\cite{Garnett,Nikolski1}.\footnote{Notice
that, when acting on $\cS'$, $P_{[0,\infty)}=P_{(0,\infty)}$ modulo
polynomials. Hence, when acting on $L^\infty$ the two operators differ
by a constant, hence coincide having values in $\BMOA$.}
We also recall that for $\vp\in L^\infty(\bbR)$, $P_+\vp$ is defined
as follows.  Let $R>0$ and $\ov D_R=\ov{D(0,R)}\subseteq\bbC$ be a fixed
compact set, and let $I=\ov D_R\cap\bbR$, $I^*= \ov D_{3R}\cap\bbR$. Let
$\vp_1=\vp\chi_{I^*}$ and $\vp_2=\vp-\vp_1=\vp\chi_{(I^*)^c}$.
Then, for $z\in \ov D_R$ given, 
\begin{equation}\label{Ppi-on-Linfty}
P_+ \vp(z) = \frac{1}{2\pi i} \operatorname{p.v.}\int_{I^*}
\vp(t)\frac{1}{t-z}\,dt + \frac{1}{2\pi i} \int_{(I^*)^c} \vp(t)\Big[ \frac{1}{t-z}-\frac1t\Big]
\, dt .
\end{equation} 
We recall that $P_+ \vp$ is well defined  on $\ov D_R$ 
up to an additive constant $c$.\footnote{For, if $R'>R$, and 
  $P_+\vp$ is defined for $z\in\ov D_{R'}$ as in
  \eqref{Ppi-on-Linfty}, then on $\ov D_R$ the two definitions differ
  by the constant
$c = \int_{I_{R'}^*\setminus I_R^*} \vp(t)/t \, dt$, where $I_{R'}^*=[-3R',3R']$.
}

\ms

We first observe that on $L^2(\bbR)$
\begin{align}
P_\pi & = e^{-i\pi(\cdot)}P_{[0,2\pi]}(e^{i\pi(\cdot)} \cdot)
= e^{-i\pi(\cdot)}\big( P_+ - P_{(2\pi,\infty)}\big)(e^{i\pi(\cdot)}\cdot)\notag\\
& = e^{-i\pi(\cdot)}\Big( P_+  - e^{2i\pi(\cdot)} 
P_+\big(e^{-2i\pi(\cdot)}\cdot) \Big) 
 (e^{i\pi(\cdot)}\cdot) \notag\\
& = e^{-i\pi(\cdot)} P_+ (e^{i\pi(\cdot)}\cdot) -  e^{i\pi(\cdot)} P_+ (e^{-i\pi(\cdot)}\cdot) .\label{1st-ob}
\end{align}
Analogously,
\begin{align}\label{2nd-ob}
P_\pi  & = e^{i\pi(\cdot)} P_- (e^{-i\pi(\cdot)}\cdot) -  e^{-i\pi(\cdot)} P_- (e^{i\pi(\cdot)}\cdot) .
\end{align}
Since $P_\pm$ extend to bounded linear operators from $L^\infty$ into
$\BMO$, we may use~\eqref{1st-ob}, or equivalently~\eqref{2nd-ob}, to
define $P_\pi$ on $L^\infty$. If $\vp\in L^\infty$, $R>0$ and $z\in\ov
D_R$, then using~\eqref{Ppi-on-Linfty} we see that
\begin{align}\label{3rd-ob}
 P_\pi \vp(z) & = e^{-i\pi z} P_+ (e^{i\pi(\cdot)}\vp)(z) -  e^{i\pi z} P_+ (e^{-i\pi(\cdot)}\vp)(z)
\notag\\
& =  \frac{1}{2\pi i}  \operatorname{p.v.}\int_{I^*}
\vp(t)\frac{e^{i\pi(t- z)} }{t-z}\,dt +  \frac{1}{2\pi i} \int_{(I^*)^c}
           \vp(t) e^{i\pi(t- z)}\Big[ \frac{1}{t-z}-\frac1t\Big]
\, dt \notag \\
& \qquad -  \frac{1}{2\pi i}  \operatorname{p.v.}\int_{I^*}
\vp(t)\frac{e^{i\pi(z-t)} }{t-z}\,dt -  \frac{1}{2\pi i} \int_{(I^*)^c}
           \vp(t) e^{i\pi(z-t)}\Big[ \frac{1}{t-z}-\frac1t\Big]
\, dt \notag \\
& = \int_{I^*} \vp(t)\sinc(t- z) \,dt +\int_{(I^*)^c}
           \vp(t) \frac{\sin(\pi(t-z))}{\pi}\Big[
                  \frac{1}{t-z}-\frac1t\Big] \,dt .
\end{align}
It is simple to check that $P_+\vp$ is defined on each closed strip
$\{z:\, |\Re z|\le R\}$,
 modulo a constant.

Notice that, since for $\vp\in
L^\infty$, $P_\pm\vp$ is defined only up to an additive constant,
$P_\pi\vp$ is defined only modulo
$\operatorname{span}\{e^{i\pi(\cdot)},e^{-i\pi(\cdot)}\}$.  It is then 
clear that $P_\pi$ maps $L^\infty$ into $\cE_\pi/ \operatorname{span}\{e^{i\pi(\cdot)},e^{-i\pi(\cdot)}\}$. \ms

We now observe that the boundary values of
functions in $H^\infty(\bbC^+)$ are multipliers of $\BMOA$, since they
are bounded multipliers of $H^1(\bbC^+)$.  In particular, $e^{\pm
  i2\pi(\cdot)}$ is a multiplier of $\BMOA^\pm$, resp., with norm $\le1$.

Let $\vp\in L^\infty$.  Then, using~\eqref{1st-ob} we see that
\begin{align*}
\| e^{i\pi(\cdot)} P_\pi \vp\|_{\BMO} 
& \le  \|P_+ (e^{i\pi(\cdot)}\vp)\|_{\BMO}  +\|  e^{2i\pi(\cdot)} P_+ (e^{-i\pi(\cdot)}\vp)\|_{\BMO}  \\
& \le  \|e^{i\pi(\cdot)}\vp\|_{L^\infty} +\|e^{-i\pi(\cdot)}\vp\|_{L^\infty}  =2 \|\vp\|_{L^\infty} .
\end{align*}
Using~\eqref{2nd-ob} we obtain that $\| e^{-i\pi(\cdot)} P_\pi
\vp\|_{\BMO} \le 2 \|\vp\|_{L^\infty} $ so that,
\begin{align*}
\| P_\pi \vp\|_{\BMO(e^{-i\pi z})}
& \le \| e^{i\pi(\cdot)} P_\pi \vp\|_{\BMO} + \| e^{-i\pi(\cdot)} P_\pi
\vp\|_{\BMO}\\
& \le 4 \|\vp\|_{L^\infty}.  
\end{align*}
Hence, $P_\pi: L^\infty\to \BMO(e^{-i\pi z})$ is bounded.  In order to show that it
is surjective, let $f\in\BMO(e^{-i\pi z})$.  Then, $e^{\pm i\pi(\cdot)}f_0\in
\BMOA^\pm$, so that there exist $\vp_\pm\in L^\infty$ such that
\begin{equation}\label{as-in-eq}
e^{ i\pi(\cdot)}f_0 = P_+ \vp_+,\quad\text{and}\quad  e^{- i\pi(\cdot)}f_0 = P_- \vp_-,.
\end{equation}
Therefore,
$$
f_0  = e^{- i\pi(\cdot)} P_+ \vp_+ =
  P_{[-\pi,\infty)}(e^{-i\pi(\cdot)}\vp_+)
  \quad\text{and also}\quad  
f_0  = e^{ i\pi(\cdot)} P_- \vp_- =
  P_{(-\infty,\pi]}(e^{i\pi(\cdot)}\vp_-).
$$
However, $\supp(\widehat{f_0})\subseteq[-\pi,\pi]$, so that the above
equalities imply that $\supp(\widehat{ e^{-i\pi(\cdot)}\vp_+})\subseteq(-\infty,\pi]$, 
and $\supp(\widehat{
  e^{i\pi(\cdot)}\vp_-})\subseteq[-\pi,\infty)$.
Therefore,
$$
f_0  = 
  P_\pi (e^{-i\pi(\cdot)}\vp_+) =
  P_\pi (e^{i\pi(\cdot)}\vp_-);
     $$
     hence $P_\pi:L^\infty \to \BMO(e^{-i\pi z})$ is onto.
     
Next, since $P_\pi:L^\infty\to \BMO(e^{-i\pi z})$ is bounded,
$$\|f\|_{\BMO(e^{-i\pi z})} \le 4\inf\{  \|\vp\|_{L^\infty}:\, \vp\in
L^\infty,\, P_\pi \vp =f\}.$$
On the other hand,
if
$g \in \BMOA$, then
$\|g\|_{\BMOA}\approx\inf\{\|\vp\|_{L^\infty}:\, \vp\in
L^\infty,\, P_+ \vp =g\}$, and analogously if $g \in \ov{\BMOA}$.
Therefore, since $e^{i\pi(\cdot)}f_0+ c_{f,+} e^{i2\pi(\cdot)} \in
\BMOA$, and $e^{-i\pi(\cdot)}f_0+ c_{f,-} e^{-i2\pi(\cdot)} \in
\ov{\BMOA}$,
there exist a constant $C_1>0$ and $\vp_\pm\in L^\infty$ 
such that~\eqref{as-in-eq} holds and such that $\|
\vp_\pm\|_{L^\infty}  \le C_1\| e^{\pm
  i\pi(\cdot)}f_0 +c_{f,\pm}
    e^{\pm i2\pi(\cdot)} \|_{\BMO(\bbR)}$. Then,
\begin{align*}
  \| e^{i\pi(\cdot)}\vp_+\|_{L^\infty}
  & = \|\vp_+\|_{L^\infty} \le C_1 \| e^{  i\pi(\cdot)}f_0+c_{f,+}
    e^{ i2\pi(\cdot)} \|_{\BMO(\bbR)}
 \le C_1 \|f\|_{\BMO(e^{-i\pi z})},
\end{align*}
as we wished to show.
\ms

We now turn to the identification of $(\cB^1_\pi)^*$ with $\BMO(e^{-i\pi z})$. Let 
$f\in\BMO(e^{-i\pi z})$, and let $\vp\in L^\infty$ such that $f=P_\pi\vp$ and
$\|\vp\|_{L^\infty}\le C_1\|f\|_{\BMO(e^{-i\pi z})}$.  Then, letting $L_f$ be
initially defined on $\cS_\pi$, we have that 
\begin{align*}
 | L_f(h) |
  & = |\int_\bbR h(x)f(x)\, dx | = |\int_\bbR h(x) P_\pi \vp (x)\, dx
    | = |\la h,P_\pi \vp\ra|\\
& = | \la \widehat{ h(-\cdot)}, \widehat{P_\pi \vp}\ra|  = | \la \widehat { h(-\cdot)}, \widehat{\vp}\ra| \\
& \le \|h\|_{L^1} \|\vp\|_{L^\infty}  \le C_1 \|h\|_{\cB^1_\pi} \|f\|_{\BMO(e^{-i\pi z})},
\end{align*}
so that $\|L_f\|_{(\cB^1_\pi)^*}\le C_1 \|f\|_{\BMO(e^{-i\pi z})}$.

Conversely, if $L\in (\cB^1_\pi)^*$, by Hahn--Banach there exists
$\vp\in L^\infty$ such that $L(h)=\int_\bbR h\vp\, dm$ for all
$h\in\cB^1_\pi$, with $\|\vp\|_{L^\infty} =\|L\|_{(\cB^1_\pi)^*}$. However,
\begin{align*}
  \int_\bbR h(x)\vp(x)\, dx & = \int_\bbR h(x)P_\pi(\vp(-\cdot))(-x)\, dx
   =      \int_\bbR h(x)P_\pi\vp(x)\, dx    =: \int_\bbR h(x)f(x)\, dx
\end{align*}
where $f =P_\pi\vp\in \BMO(e^{-i\pi z})$.  Then,
$L=L_f$.  Moreover, 
$$
\|f\|_{\BMO(e^{-i\pi z})} = \|P_\pi\vp\|_{\BMO(e^{-i\pi z})} \le 4 \|\vp\|_{L^\infty} =
4\|L_f\|_{(\cB^1_\pi)^*}.
$$
This proves the theorem.\qed
\ms

\section{Proof of Theorem~\ref{main-thm2}}\label{Pf-Thm2:sec}

 Recall the definition~\eqref{Hankel-op:eq} of
Hankel operator $H_\vp$ on
$\cB^2_\kappa$, having as symbol the locally integrable function $\vp$
which is also a tempered distribution 
on $\bbR$.  Notice that $H_\vp$ is well defined on the dense subset
$\cS_\kappa$ of $\cB^2_\kappa$.  We wish to characterize the symbols
$\vp$ for which $H_\vp$ extends to a bounded linear operator on 
$\cB^2_\kappa$. 
For general facts about Hankel operators on Hardy spaces we
refer to~\cite{Nikolski1}.

The space $\cB^2_\kappa$ is endowed with the natural involution
$f\mapsto f^\#$, where $f^\#(z)=\ov{f(\ov z)}$. Hence, for $x\in\bbR$, $f_0^\#(x)=\ov{f_0(x)}$. Then,
 for
$f\in\cB^2_\kappa$, arguing formally for the time being, we have that 
\begin{align*}
  H_\vp f(z) & = \cF^{-1} \big( \chi_{[-\kappa,\kappa]}\cF(\vp f_0^\#)\big)(z)\\
  & = \frac{1}{\sqrt{2\pi}} \int_{-\kappa}^\kappa e^{iz\xi}
            (\widehat{\vp}* \widehat{f_0^\#} )(\xi)\, d\xi\\ 
& = \frac{1}{\sqrt{2\pi}} \int_{-\kappa}^\kappa e^{iz\xi} \int_{-\kappa}^\kappa 
\widehat{\vp}(\xi-t)\widehat{f_0^\#}(t)\, dt \, d\xi.
\end{align*}
Therefore, the operator $H_\vp$ remains unchanged if we replace $\vp$
with $\cF^{-1}\big(\chi_{[-2\kappa,2\kappa]}\widehat{\vp}\big)$.
Therefore, we may assume, and we do so in what follows, that $\supp(\widehat\vp)\subseteq
[-2\kappa,2\kappa]$.\footnote{We recall that it is not in general true that
$H_\vp=H_{P_\kappa \vp}$ since $\cE_\kappa$ is not an algebra.}
\ms

Let now $\eta_L\in C^\infty_c(\bbR)$ be such that $\supp\eta\subseteq
[-4\kappa,-\kappa/2]$, $\eta\ge0$,  and $\eta=1$ on
$[-3\kappa, -\kappa]$.  Let $\eta_R=\eta_L(-\cdot)$ and $\eta_C=
\chi_{[-2\kappa,2\kappa]}-\eta_L-\eta_R$.  Define then
$\vp_L=\cF^{-1}(\eta_L\widehat\vp)$, and analogously
$\vp_R=\cF^{-1}(\eta_R \widehat\vp)$ and
$\vp_C=\cF^{-1}(\eta_C\widehat\vp)$.
We say that a symbol $\psi$ is {\em analytic} if
$\widehat\psi$ is supported in $[0,\infty)$, and analogously, that it
is {\em anti-analytic} if $\widehat\psi$ is supported in
$(-\infty,0]$.
Then, $\vp_L$ is anti-analytic and $\vp_R$ analytic.

In~\cite[Theorem 4.1]{Rochberg} it is proved that $H_\vp$ is bounded
if and only if
\begin{itemize}
  \tbi $P_- \big( e^{-2i\kappa(\cdot)}\vp_L\big) \in \BMO(\bbR)$;
  \tbi $P_+ \big( e^{2i\kappa(\cdot)}\vp_R\big) \in \BMO(\bbR)$;
  \tbi $\vp_C\in L^\infty (\bbR)$.
\end{itemize}
Moreover,
$$
\| H_\vp\|_{\cB^2_\kappa\to\cB^2_\kappa} \approx \|P_- \big(
e^{-2i\kappa(\cdot)}\vp_L\big)\|_{\BMO} + \|P_+ \big(
e^{2i\kappa(\cdot)}\vp_R\big)\|_{\BMO} +\|\vp_C\|_{L^\infty}. \ms
$$

Another key fact we are going to use is a factorization decomposition
of $\cB^1_\kappa$, obtained by A.\ Baranov, R.\  Bessonov, and V.\
Kapustin~\cite[Proposition 4.1, and Section 7]{Baranov-Bessonov-Kapustin} (as corollary of a more
general result).  Given any $h\in \cB^1_{2\kappa}$, there exist
$f_j,g_j\in\cB^2_\kappa$, $j=1,\dots,4$, such that
$h =\sum_{j=1}^4 f_jg_j$
and
\begin{equation}\label{fact-B1}
\|h\|_{\cB^1_{2\kappa}}\approx \inf \Big\{ \sum_{j=1}^4
\|f_j\|_{\cB^2_\kappa}\|g_j\|_{\cB^2_\kappa}:\, h =\sum_{j=1}^4 f_jg_j  \Big\}.
\end{equation}
It is easy to see that if $h\in\cS_{2\kappa}$, it is possible to
select $f_j,g_j \in \cS_\kappa$ such that~\eqref{fact-B1}
holds. \ms

\proof[Proof of Theorem~\ref{main-thm2}] 
We are going to use the above mentioned results with $\kappa=\pi/2$.
Observe that if $\vp$ is as in 
(2), since we are assuming, as we may, that
$\supp(\widehat\vp)\subseteq[-\pi,\pi]$, 
 there exists $f\in\cE_\pi$ such that $\vp=f_0$.
Notice that Theorem~\ref{main-thm1bis} shows that (1) and (2) are
equivalent.\ms

$(3)\Rightarrow (2)$.  Given $h\in\cS_\pi$, let $h =\sum_{j=1}^4
f_jg_j$ be as in~\eqref{fact-B1}, with $f_j,g_j\in\cB^2_{\pi/2}$.  By Lemma~\ref{density-lem}, without loss of
generality we may assume $f_j,g_j$ to be in
$\cS_{\pi/2}$. Then,
\begin{align*}
\la h , \vp\ra & = \big\la \sum_{j=1}^4f_jg_j  ,\vp \big\ra\\
& =\sum_{j=1}^4 \int_\bbR f_j(x)  g_j(x) \vp(x) \,dx =
 \sum_{j=1}^4 \int_\bbR f_j(x)  \ov{P_{\pi/2} (\ov{g_j} \ov{\vp})(x)} \,dx  \\
& = \sum_{j=1}^4 \int_\bbR f_j(x)  \ov{H_{\ov\vp} g_j(x)}\, dx.
\end{align*}
Therefore,
$$
|\la h ,\vp\ra | \le \sum_{j=1}^4 \|f_j\|_{\cB^2_{\pi/2}}\| H_{\ov\vp} \|_{\cB^2_{\pi/2}\to\cB^2_{\pi/2}}
\|g_j\|_{\cB^2_{\pi/2}} \le C_2 \|h\|_{\cB^1_\pi} \| H_{\ov\vp} \|_{\cB^2_{\pi/2}\to\cB^2_{\pi/2}},
$$
so that $\vp\in (\cB^1_\pi)^*$, and  moreover $\|\vp\|_{(\cB^1_\pi)^*}\le C_2\| H_{\ov\vp}
\|_{\cB^2_{\pi/2}\to\cB^2_{\pi/2}}$.  We observe that
$$
H_{\ov\vp} f = H_\vp f^\#
$$
and since $f\mapsto f^\#$ is an isometry on the Paley--Wiener (and
Bernstein) spaces, $H_{\ov\vp}$ is bounded if and only if $ H_\vp $
is, with equality of norms.
\ms

$(1)\Rightarrow (3)$.  If $f,g\in\cB^2_{\pi/2}$, then
\begin{align*}
 | \la H_\vp f \,|\, g \ra |
 & = |\la \vp \,|\, fg \ra  | = | \la P_\pi \widetilde{\vp}\,|\, fg\ra  | 
= | \la  \widetilde{\vp}\,|\, fg\ra  | \le \|\widetilde{\vp}\|_{L^\infty}
\|fg\|_{\cB^1_\pi}\\
& \le \|\widetilde{\vp}\|_{L^\infty}\|f\|_{\cB^2_{\pi/2}}\|g\|_{\cB^2_{\pi/2}}.
\end{align*}
Therefore, $H_\vp$ is bounded on $\cB^2_{\pi/2}$ and $\| H_\vp
\|_{\cB^2_{\pi/2}\to\cB^2_{\pi/2}}\le \inf\{
\|\widetilde{\vp}\|_{L^\infty}:\, \vp =P_\pi \widetilde{\vp},\,
\widetilde{\vp}\in L^\infty\}$. \ms

Finally, we have shown that
$$
\|\vp\|_{(\cB^1_\pi)^*}\le C_2\| H_\vp \|_{\cB^2_{\pi/2}\to\cB^2_{\pi/2}}
\le  C_2\inf\{ \|\widetilde{\vp}\|_{L^\infty}:\, \vp =P_\pi \widetilde{\vp},\, \widetilde{\vp}\in L^\infty\} 
.
$$
From Theorem~\ref{main-thm1bis} we also have that $\vp=f_0$ for some
$f\in\BMO(e^{-i\pi z})$ and
$$
\|\vp\|_{(\cB^1_\pi)^*} \approx \|f\|_{\BMO(e^{-i\pi z})} \approx 
\inf\{ \|\widetilde{\vp}\|_{L^\infty}:\, f =P_\pi \widetilde{\vp},\, \widetilde{\vp}\in L^\infty\} .
$$
The stated norm equivalences now follow.
\qed \ms

\section{Discrete nature of the pairing of duality}\label{Pf-Thm1:sec}

We now turn to the discrete characterization of $(\cB^1_\pi)^*$.
We begin with some properties of
$\cX_\alpha$.

\begin{prop}\label{Xalpha-Banach:prop}
   For $\alpha\in[0,1)$, $\cX_\alpha$  is a Banach space. 
   \end{prop}

  \proof We have observed that $\cX_\alpha$ is well defined.  In order
  to prove that the given one is a norm, we only need to prove the
  triangular inequality and that if $\|f\|_{\cX_\alpha}=0$, then
  $f=0$.
  
Observe that $\operatorname{span}\{e^{\pm
        i\pi(\cdot)} \}=
      \operatorname{span}\{\cos\pi(\cdot),\sin\pi(\cdot) \}$ and that
\begin{align}
  \|f\|_{\cX_\alpha}
  & = \inf_{c_\pm\in\bbC} \big\|
      e^{i\pi(\cdot+\alpha)} \big( f(\cdot+\alpha)
  +c_+ e^{i\pi(\cdot)}+c_- e^{-i\pi(\cdot)}\big) \big\|_{\BMO(\bbZ)}\notag\\
&  = \big\|
      (-1)^{(\cdot)} \big( f(\cdot+\alpha)
  \big)\big\|_{\BMO(\bbZ)}. \label{norm-id-Xalpha}
\end{align}
  Suppose that $\| f\|_{\cX_\alpha} =0$. Then, 
$\big( (-1)^{(\cdot)} \big( f(\cdot+\alpha)\big)=
 0$ in $\BMO(\bbZ)$.
It is then easy to see that (for every representative of $f$ in its
equivalence class),
there exist constants $a,b\in\bbC$ such
that 
if we set
$$
F(z):= f(z+\alpha) + a\cos(\pi z) + b\sin(\pi z)
$$
then $F$ vanishes at all points $n\in\bbZ$.  Moreover, any such $F$ satisfies condition
 $(Y)$ of Lemma~\ref{exp-growth:lem} and therefore $F(z) =c\sin(\pi
z)$ for some $c\in\bbC$. This implies that
$f(\cdot+\alpha)\in\operatorname{span}\{e^{\pm i(\cdot)}\}$, that is, $f=0$ in
$\cX_\alpha$. 

Finally, given $f,g\in\cX_\alpha$, by~\eqref{norm-id-Xalpha} we have that
\begin{align*}
\| f+g\|_{\cX_\alpha} & =\big\|
 (-1)^{(\cdot)} \big( f(\cdot+\alpha) +g(\cdot+\alpha) \big) \big\|_{\BMO(\bbZ)}\\
& \le \big\|
 (-1)^{(\cdot)} \big( f(\cdot+\alpha) \big) \big\|_{\BMO(\bbZ)} + \big\|
 (-1)^{(\cdot)} g(\cdot+\alpha) \big)
\big\|_{\BMO(\bbZ)}\\
& =  \| f\|_{\cX_\alpha} +\| g\|_{\cX_\alpha} .
\end{align*}
This proves the Proposition. \qed
\ms

Recall that, see e.g.~\cite{Coifman-Weiss}, if $a\in \BMO(\bbZ)$, then
\begin{equation}\label{BMO-Z-estimate}
\sum_{n\in\bbZ} \frac{|a_n|}{1+n^2}<\infty.
\end{equation}  

\proof[Proof of Theorem~\ref{main-thm1}]

{\em Step 1.} The mapping $T_\alpha$ is well defined on $ \BMO(\bbZ)$
and if $a\in \BMO(\bbZ)$, $T_\alpha a$ is an entire function.
 \ms

Let $\mathbf{1}=(c_n)_{n\in\bbZ}$ be the constant sequence with $c_n=1$ for all
$n\in\bbZ$.  Then, we have
\begin{align*}
e^{i\pi\alpha}  T_\alpha \mathbf{1}(z)
  & = \sum_{n\in\bbZ\setminus\{0\}} (-1)^n  \Big(
  \frac{1}{z-\alpha-n} +\frac{1}{n}\Big)
  \frac{\sin[\pi(z-\alpha-n)]}{\pi} +   \sinc(z-\alpha)\\
  & = e^{i\pi\alpha} \cos(\pi(z-\alpha)),
\end{align*}
by a standard identity, see e.g.~\cite[Example 2.4, Ch. XIII]{Lang}. Hence, $ T_\alpha
\mathbf{1}=0$ in $\cX_\alpha$, that is, $T_\alpha$ is well defined on
$\BMO(\bbZ)$. 

Let now $a=(a_n)\in \BMO(\bbZ)$, and choose the representative sequence such
that $a_0=0$. Fix $R>1$ and the compact set $\ov{D(0,R)}$.  Then,
\begin{align*}
 e^{i\pi\alpha} T_\alpha a(z)
  & =   \sum_{n\in\bbZ\setminus\{0\}}  a_n \Big(
  \frac{1}{z-\alpha-n} +\frac{1}{n}\Big)
  \frac{\sin[\pi(z-\alpha)]}{\pi} \\
&    =\Big(\sum_{0<|n|\le 2R}+ \sum_{|n|> 2R}\Big) a_n 
  \frac{z-\alpha}{(z-\alpha-n)n} 
  \frac{\sin[\pi(z-\alpha)]}{\pi} \\
  & =: S_1(z)+ S_2(z).
\end{align*}

For $z\in\ov{D(0,R)}$ and $|n|>2R$ we observe that
\begin{align*}
\Big| a_n 
  \frac{z-\alpha}{(z-\alpha-n)n} 
  \frac{\sin[\pi(z-\alpha)]}{\pi} \Big| & \le \sup_{|z|\le 2R}  \frac1\pi\big|(z-\alpha)\sin[\pi(z-\alpha)]\big|
\frac{|a_n|}{|n|(|n|-R-1)|}\\ 
& \le C_R \frac{|a_n|}{|n|(|n|+1)},
\end{align*}  
so that the sum defining $S_2$ converges uniformly on
$\ov{D(0,R)}$. On the other hand,
if $z\in\ov{D(0,R)}$ and $|n|\le2R$ we have that
\begin{align*}
\Big| a_n 
  \frac{z-\alpha}{(z-\alpha-n)n} 
  \frac{\sin[\pi(z-\alpha)]}{\pi} \Big| & = |z-\alpha||\sinc(z-\alpha-n)|\frac{|a_n|}{|n|}\\
& \le C'_R\frac{|a_n|}{|n|}
\end{align*}  
and since $S_1$ is given by a finite sum, the convergence is uniform.
Thus, $T_\alpha a$ defines an entire function.
\ms

{\em Step 2.} For $a\in \BMO(\bbZ)$, $T_\alpha a\in \cE_\pi$. \ms

We slipt the sum defining $T_\alpha a$ as follows:
\begin{align*}
 e^{i\pi\alpha} T_\alpha a(z)
&    =\Big(\sum_{0<|n|\le 2|z|}+ \sum_{|n|> 2|z|}\Big) a_n 
  \frac{z-\alpha}{(z-\alpha-n)n} 
  \frac{\sin[\pi(z-\alpha)]}{\pi} \\
  & =: \widetilde S_1(z)+ \widetilde S_2(z).
\end{align*}
Using~\eqref{BMO-Z-estimate} we see that
\begin{align*}
|\widetilde S_2(z)|
& \le \sum_{|n|>2|z|} \frac1\pi\big|(z-\alpha)\sin[\pi(z-\alpha)]\big|
\frac{|a_n|}{|n|(|n|-R-1)|}\\
& \le C_R \frac1\pi\big|(z-\alpha)\sin[\pi(z-\alpha)]\big|\sum_{|n|>2|z|} \frac{|a_n|}{|n|(|n|+1)|}\\
  & \le C'_R \frac1\pi\big|(z-\alpha)\sin[\pi(z-\alpha)]\big|.
\end{align*}
Hence, $\widetilde S_2\in \cE_\pi$. Next, writing $z=x+iy$,
\begin{align*}
|\widetilde S_1(z)|
  & \le  \sum_{0<|n|\le 2|z|} |z-\alpha||\sinc(z-\alpha-n)|\frac{|a_n|}{|n|}
   \\
  & \le e^{\pi|y|} |z-\alpha| \sum_{0<|n|\le 2|z|} 2|z|\frac{|a_n|}{|n|^2}. 
\end{align*}
This shows that $T_\alpha a\in\cE_\pi$. \ms

{\em Step 3.} We show that for $a\in \BMO(\bbZ)$, $T_\alpha a(iy) =o(|y|e^{\pi|y|})$.
\ms

We have
\begin{align*}
 \lim_{|y|\to\infty} \frac{|T_\alpha a(iy)|}{|y|e^{\pi|y|}}
& = \lim_{|y|\to\infty} \frac{|iy -\alpha||\sinc(z-\alpha-n)|}{|y|e^{\pi|y|}}\Big| 
\sum_{n\neq0} \frac{a_n}{(iy-\alpha-n)n}\Big| ,
\end{align*}
so it suffices to show that
$$
\lim_{|y|\to\infty} \sum_{n\neq0} \frac{a_n}{(iy-\alpha-n)n} =0.
$$
This follows at once from the dominated convergence theorem, whence
$T_\alpha a(iy) =o(|y|e^{\pi|y|})$.\ms

{\em Step 4.} We now show that $T_\alpha a \in \cX_\alpha$.
\ms

When $k\neq0$ we have that
\begin{align}
& e^{i\pi(k+\alpha)}T_\alpha a(k+\alpha)\notag \\
  & = (-1)^{k}\bigg( \sum_{n\in\bbZ\setminus\{0,k\}} (-1)^n a_n \Big(
  \frac{1}{k-n} +\frac{1}{n}\Big)
  \frac{\sin[\pi(k-n)]}{\pi} +(-1)^k a_k+ (-1)^{k}a_0 \sinc(k) \bigg) \notag\\
  & =   a_k. \label{Talpha-iden}
\end{align}
The case $k=0$ is immediate, and therefore
$$
\big( e^{i\pi(\cdot+\alpha)}T_\alpha a(\cdot+\alpha)\big)_{|_\bbZ} = a \in \BMO(\bbZ).
$$
This fact, together with the conclusions of {\em Steps 1-3} imply that
 $T_\alpha\in\cX_\alpha$,
 as we wished to show. \ms

{\em Step 5.}  The operator $T_\alpha$ is an isomorphism. \ms

Suppose $T_\alpha a=0$, that is, $T_\alpha a(z)=c_1 \cos(\pi(
z-\alpha))+c_2\sin(\pi (z-\alpha))$, for some  $c_1,c_2\in\bbC$. In {\em Step 1} we have observed
that $T(c_1{\mathbf 1})=c_1 \cos(\pi (z-\alpha))$.  Since, in $\BMO(\bbZ)$,
$a=a+c_1{\mathbf 1}$, we have that $T_\alpha a(z)=c_2\sin(\pi
(z-\alpha))$. By~\eqref{Talpha-iden} it then follows that
$$
a_k = e^{i\pi(k+\alpha)}T_\alpha a(k+\alpha) = 0.
$$
Hence, $a_k=0$ for all $k\in\bbZ$ so that the equality holds with
$c_2=0$ and
$T_\alpha$ is injective.

Next, let $f\in\cX_\alpha$.  Then
$a:=\big((-1)^{(\cdot)}f(\cdot+\alpha)\big)\in\BMO(\bbZ)$. Therefore,
$f-T_\alpha a\in\cX_\alpha$ and
$
f(n+\alpha)-T_\alpha a(n)=0$ for every $n\in\bbZ$.  By
Lemma~\ref{exp-growth:lem}, there exists $c\in\bbC$ such that
$f=T_\alpha a+c\sin\pi(\cdot)$, that is, $T_\alpha a=f$ in
$\cX_\alpha$, i.e. $T_\alpha$ is also onto. \ms

{\em Step 6.} We now wish to show that each of the spaces $\cX_\alpha$,
$\alpha\in[0,1)$ can be identified with the dual space
$(\cB^1_\pi)^*$.\ms

From~\cite{Eoff} we know that $\widetilde
T_\alpha:H^1(\bbZ)\to\cB^1_\pi$ is an isomorphism, where
$$
\widetilde T_\alpha a(z)= \sum_{n\in\bbZ} a_n (-1)^n \sinc(z-\alpha-n).
$$
where $\widetilde T_\alpha^{-1}$ is given by
$$
\widetilde T_\alpha^{-1}f =\big( e^{i\pi\alpha} (-1)^{(\cdot)}f(\cdot+\alpha) \big)_{|_\bbZ}.
$$
Then, $(\widetilde T_\alpha^{-1})^*:\BMO(\bbZ)\to(\cB^1_\pi)^*$ is an
isomorphism, where, if $f\in \cB^1_\pi$ and $b\in\BMO(\bbZ)$, the
pairing of duality is given by
\begin{align*}
(\widetilde T_\alpha^{-1})^* b (f) & = \la \widetilde T_\alpha^{-1}f ,b \ra
= \la b, \big( e^{i\pi\alpha} (-1)^{(\cdot)}f(\cdot+\alpha) \big)_{|_\bbZ}\ra\\
 & = \la \big(  e^{i\pi\alpha} (-1)^{(\cdot)}T_\alpha b(\cdot+\alpha)\big)_{|_\bbZ},\,
\big( e^{i\pi\alpha} (-1)^{(\cdot)}f(\cdot+\alpha) \big)_{|_\bbZ}\ra\\
& = \la e^{2i\alpha \pi} \big( T_\alpha b(\cdot+\alpha)\big)_{|_\bbZ},\,
\big( f(\cdot+\alpha) \big)_{|_\bbZ}\ra
\end{align*}
This proves the theorem.
\qed

\begin{cor}\label{Xalpha-iso:cor}
 For each $\alpha\in(0,1)$, the space
   $\cX_\alpha$ coincides with $\cX_0$, with equivalence of norms.  Hence, 
the spaces $\cX_\alpha$ are all equal. 
\end{cor}
Let $f\in\cX_0$. In {\em Step 6} we have shown that $T_0 a=f$,
where $a:=\big((-1)^{(\cdot)}f(\cdot)\big)\in\BMO(\bbZ)$, so that
\begin{align*}
  (-1)^k f(k+\alpha) & = (-1)^k
\bigg( \sum_{n\neq0} (-1)^n a_n \Big(
   \frac{1}{k+\alpha-n}+\frac1n  \Big) \frac{\sin(\pi(k+\alpha-n))}{\pi}
 + a_0 \frac{\sin(\pi(k+\alpha))}{\pi(k+\alpha)} \bigg)\\
  & =   (-1)^k \frac{\sin(\pi(k+\alpha))}{\pi} \sum_{n\neq0}  a_n \Big(
   \frac{1}{k+\alpha-n}+\frac1n  \Big) 
 + (-1)^ka_0 \sinc(k+\alpha)  \\
& = \frac{e^{i\pi\alpha}}{\pi} \big(H_{d,\alpha} a\big)_k
 + (-1)^ka_0 \sinc(k+\alpha) .                       
\end{align*}
Therefore, possibly replacing $f$ by $f-a_0\sinc(\cdot)$, or assuming
that $a_0=0$ as we also may, we obtain that
$$
R_\alpha f:= e^{-i\pi\alpha}\big( (-1)^{(\cdot)}
f(\cdot+\alpha)\big)= \frac1\pi H_{d,\alpha} R_0f \in\BMO(\bbZ).
$$
Since $H_{d,\alpha}:\BMO(\bbZ)\to\BMO(\bbZ)$ is an isomorphism, 
it follows that
$\big((-1)^{(\cdot)}f(\cdot+\alpha)\big)\in\BMO(\bbZ)$, and
\begin{align*}
  \|f\|_{\cX_\alpha}
  & \le  \big\| e^{-i\pi\alpha}\big( (-1)^{(\cdot)}
    f(\cdot+\alpha)\big)\big\|_{\BMO(\bbZ)}\\
  & \le C_\alpha \|R_0f\|_{\BMO(\bbZ)} \\
  & \le C_\alpha \|f\|_{\cX_0}.
\end{align*}
This shows that 
$\cX_0$  continuously embeds into $\cX_\alpha$.

Since $H_{d,1-\alpha}\circ H_{d,\alpha} =-S$, where
$
S(a)(\cdot):=a(\cdot+1),
$
and $S$ is an isomorphism in $\BMO(\bbZ)$,
 from the preceeding formula it follows that 
$R_\alpha$ is one-to-one and $R_\alpha^{-1}=-{R_0}^{-1}S^{-1}H_{d,1-\alpha}$. Then, the
embedding is also onto.

The
equivalence of the norms now follows from the open mapping theorem.
\ms \qed

\proof[Proof of Theorem~\ref{Y=Xalpha:thm}]
Let $f\in\BMO(e^{-i\pi z})$, that we can identify with $(\cB^1_\pi)^*$.  Then, the
Hankel operator with symbol $f$,
$H_f: \cB^2_{\pi/2}\to
  \cB^2_{\pi/2}$ is bounded.  This in turn is equivalent to saying
  that the Toeplitz operator $T_f: \cB^2_{\pi/2}\to
  \cB^2_{\pi/2}$ is bounded.  By~\cite{Carlsson} it follows that
  $\big((-1)^nf(n)\big) \in\BMO(\bbZ)$.  In order to show that $f$
  satisfies condition $(Y)$, we use Theorem~\ref{main-thm2} and
  using~\eqref{3rd-ob} we write $f=P_\pi
  \vp$, with $\vp\in L^\infty$, so that if $R>0$ is given, $z\in\bbC$ such that $|\Re
  z|\le R$, then
\begin{align*}
  |f(iy)| &
           \le \int_{I^*} |\vp(t)| \frac{2e^{\pi|y|}}{|y|} \,dt +C \int_{(I^*)^c}
           |\vp(t)| \frac{e^{\pi|y|}|y|}{|t|(|t|+|y|)}  \,dt 
\end{align*}
Using the dominated convergence theorem one sees at once that $f$
satisfies condition 
$(Y)$.  Therefore, $f\in\BMO(e^{-i\pi z})$ belongs to $\cX_0$, and $\BMO(e^{-i\pi z})$
continuously embeds in $\cX_0$. 
  \qed

\ms

\section{Compactness, $\VMO$ and preduality}\label{VMO:sec}

\proof[Proof of Theorem~\ref{main-thmVMO}]
(1) Let $\eta\in C_0(\bbR)$ then by \eqref{1st-ob} and \eqref{2nd-ob} we have
\begin{align*}
P_\pi\eta
& =  e^{-i\pi(\cdot)} P_+ (e^{i\pi(\cdot)}\eta) -  e^{i\pi(\cdot)} P_+ (e^{-i\pi(\cdot)}\eta) \\
 & = e^{i\pi(\cdot)} P_- (e^{-i\pi(\cdot)}\eta) -  e^{-i\pi(\cdot)} P_- (e^{i\pi(\cdot)}\eta) .
\end{align*}
Since $P_\pm$ map $C_0(\bbR)$ onto $\VMOA^\pm$, resp., we immediately
see that $P_\pi \eta\in\VMO(e^{-i\pi z})$.    In order to show that
$P_\pi(C_0(\bbR))=\VMO(e^{-i\pi z})$, 
 let $f\in\cY_0$.  Then, $e^{\pm i\pi(\cdot)}f_0\in
\VMOA^\pm$, resp., so that there exist $\eta_\pm\in C_0(\bbR)$ such that
\begin{equation}\label{as-in-eq:compactness}
e^{ i\pi(\cdot)}f_0 = P_+ \eta_+,\quad\text{and}\quad  e^{- i\pi(\cdot)}f_0 = P_- \eta_-,.
\end{equation}

Arguing as in the case of identity \eqref{as-in-eq}, we have that
$$
f_0  = 
  P_\pi (e^{-i\pi(\cdot)}\eta_+) =
  P_\pi (e^{i\pi(\cdot)}\eta_-);
  $$
  hence $P_\pi(C_0(\bbR))=\VMO(e^{-i\pi z})$.
  The stated about the equivalence of norms follows now from Theorem~\ref{main-thm1bis}.
  \ms
  
  (2)  A function $h\in\cB^1_\pi$ clearly defines an
  element $L_h$ of ${\VMO(e^{-i\pi z})}^*$ with the same pairing between $\cB^1_\pi$
  and $\BMO(e^{-i\pi z})$, as in formula \eqref{pairing-duality}.
  
  Next, let $L\in\cY_0^*$ and consider the bounded linear functional
  $\Phi$  on $C_0(\bbR)$,
  $\Phi = L\circ P_\pi  $.  Then, by the Riesz representation theorem
  there exists a unique finite regular Borel measure $\mu$ on $\bbR$ that
  represents the functional, that is,
  $$
L(P_\pi\eta)= \int_\bbR \eta\, d\mu.
  $$
  We wish to show that
$\mu$ has a density, which coincides with the restriction of a
function $\cB^1_\pi$, that is, 
  $\mu=h_0dm$, with $h\in\cB^1_\pi$. 
  Observe that if $\supp(\widehat\eta)$ does not intersect $[-\pi,\pi]$
  so that
  $P_\pi\eta=0$, then $\int_\bbR \eta\, d\mu=0$.  Hence, the
  distribution $\widehat\mu$ is
  supported on $[-\pi,\pi]$.
  By the classical Paley--Wiener--Schwartz theorem, there exists
a function $h$ in $\cE_\pi$ such that $\mu=h_0$, as distributions. In
particular, for all Schwartz functions 
$\widetilde\eta$, $\la \widetilde\eta,\mu\ra =\int_\bbR
\widetilde\eta h_0 dm$.
  Approximating $\eta\in C_0(\bbR)$ with Schwartz functions
  $\eta^{(\delta)}$ and
  passing to the limit we have
\begin{align*}  
\int_\bbR \eta\, d\mu
  & = \lim_{\delta\to0} \int_\bbR \eta^{(\delta)} \, d\mu
    = \lim_{\delta\to0}  \int_\bbR
\eta^{(\delta)} (x) h_0(x) \,dx =\int_\bbR
\eta (x) h_0(x)\, dx.
\end{align*}
Hence, $d\mu=h_0dm$ and since $\mu$ has bounded total variation,
$h_0\in L^1(\bbR)$, that is, $h\in \cB^1_\pi$.  Since $P_\pi h=h$,
using part (1) and a
simple approximation argument, we see that
$$
L(\vp)=L(P_\pi\eta)=  \int_\bbR
\eta (x) h_0(x)\, dx = \int_\bbR P_\pi\eta (x) h_0(x)\, dx, 
$$
that is, $L(\vp) = L_h(\vp)$, for all $\vp\in\cY_0$, as we wished to
show.

By part (1), there exists $C_1>0$ such that for every $\vp\in\cY_0$,
there exists $\eta_\vp\in C_0(\bbR)$ such that $P_\pi \eta_\vp=\vp$
and $\|\eta_\vp\|_{L^\infty}\le C_1 \|\vp\|_{\VMO(e^{-i\pi z})}$.  Then, for
$h\in\cB^1_\pi$,
\begin{align*}
  \|L_h\|_{{\VMO(e^{-i\pi z})}^*}
  & =\sup_{\|\vp\|_{\VMO(e^{-i\pi z})}\le1} |L_h(\vp)| \\
  &=\sup_{\|\vp\|_{\VMO(e^{-i\pi z})}\le1} |\int_\bbR h_0\eta_\vp \, dm| \le  \sup_{\|\vp\|_{\VMO(e^{-i\pi z})}\le1} 
    C_1\|h\|_{\cB^1_\pi} \|\vp\|_{\VMO(e^{-i\pi z})},
\end{align*}
whence $\|L_h\|_{{\VMO(e^{-i\pi z})}^*}\le  C_1\|h\|_{\cB^1_\pi} $.  On the other
hand, if $h\in\cB^1_\pi$, 
\begin{align*}
  \|h\|_{\cB^1_\pi}
  & =\sup_{\eta\in C_0(\bbR),\, \|\eta\|_{L^\infty} \le1} |\int_\bbR
    \eta(x) h_0(x)\, dx| =
   \sup_{\eta\in C_0(\bbR),\, \|\eta\|_{L^\infty} \le1}
    |L_h(P_\pi\eta)|\\
  & \le
\sup_{\eta\in C_0(\bbR),\, \|\eta\|_{L^\infty} \le1}  \|L_h\|_{{\VMO(e^{-i\pi z})}^*} \|P_\pi\eta\|_{\VMO(e^{-i\pi z})}
\le C_2  \|L_h\|_{{\VMO(e^{-i\pi z})}^*} .
\end{align*}
This completes the proof.
\qed\ms 

We point out that in the course of the proof, we used  the following analouge of
the theorem of F. and M. Riesz.

\begin{prop}\label{Riesz's-thm}
Let $\mu$ a finite Borel measure on $\bbR$ such that
$\supp(\widehat\mu)\subseteq[-\pi,\pi]$. Then, there exists
$h\in\cB^1_\pi$ such that $d\mu= h_0dm$. In particular, $\mu$ is absolutely
continuous w.r.t. the Lebesgue measure.
  \ms
\end{prop}

\proof[Proof of Theorem~\ref{main-thm3}]
Let $\vp\in\cY_0$, $\vp=P_\pi \eta$ with 
$\eta\in C_0(\bbR)$, be given.
 Let $\{f_n\}$ be a sequence in $\cB^2_{\pi/2}$ weakly
converging to $0$.  It is well known, and easy to see, that this is
equivalent to the conditions
$\|f_n\|_{\cB^2_{\pi/2}}\le C$, with $C$ independent of $n$ and $f_n\to
0$ uniformly on compact subsets\footnote{ We provide the details in the subsequent Lemma \ref{da-togliere}. }.
Given $\eps>0$, let $E\subseteq\bbR$ compact  and $n_\eps$
positive integer
be such that
$$
\sup_{x\in E^c} |\eta(x)|<\eps,\quad\text{and}\quad \sup_{x\in E, n\ge
n_\eps} |f_n(x)|<\eps/|E|^{1/2}.
$$
Notice that, since $f_ng\in \cB^1_\pi$,
\begin{align*}
|\la H_\vp f_n \,|\, g\ra |
& = |\la \vp  \,|\, f_ng\ra| = |\la \eta \,|\,  f_ng\ra| \le \bigg(\int_E +\int_{E^c}\bigg)
|\eta(x)f_n(x)g(x)|\, dx \\
& \le \|\eta\|_{L^\infty} \eps  \|g\|_{\cB^2_{\pi/2}}  + C \eps \|g\|_{\cB^2_{\pi/2}} \\
& \le C' \eps\|g\|_{\cB^2_{\pi/2}}.
\end{align*}  
Then, for $n\ge n_\eps$ we have
\begin{align*}
\| H_\vp f_n\|_{\cB^2_{\pi/2}}
& = \sup_{g\in\cB^2_{\pi/2},\|g\|_{\cB^2_{\pi/2}}\le1} \big|\la H_\vp f_n \,|\,  g\ra \big|
\le C' \eps  .
\end{align*}
Hence, $(1)$ implies $(2)$. \ms

Next, let
$$
\cH:= \big\{ \psi\in\cS':\, H_\psi \ \text{is a compact Hankel
  operators on} \ \cB^2_{\pi/2} \big\}.
$$
The first part of the proof and the norm equivalence statement
in Theorem~\ref{main-thm2} show that the mapping
$J:\VMO(e^{-i\pi z})\ni\vp\mapsto H_\vp\in\cH$ is an isomorphism onto its
image. Thus, $J(\VMO(e^{-i\pi z}))$ is a closed subspace of $\cH$.  If we show
that $\cH$ and $J(\VMO(e^{-i\pi z}))$ have the same dual, it would follow that
they are actually equal.  We use the identification of
${\VMO(e^{-i\pi z})}^*$ with $\cB^1_\pi$.

If $f\in\cB^1_\pi$, then define a linear
functional on $\cH$ by setting
$$
\Lambda_f(H_\vp) :=  \la \vp \,|\,f\ra  .
$$
Note that $\Lambda_f$ is well defined since, if $f=\sum_{j=1}^4
g_jh_j$, with $g_j,h_j\in \cB^2_{\pi/2}$ and
$\sum_{j=1}^4\|g_j\|_{\cB^2_{\pi/2}}\|h_j\|_{\cB^2_{\pi/2}}\le
C_1\|f\|_{\cB^1_\pi}$, then 
\begin{align*}
|\Lambda_f(H_\vp) |
& = | \sum_{j=1}^4 \la \vp \ov{g_j}\,|\, h_j \ra| = |\sum_{j=1}^4\la H_\vp g_j\,|\, h_j \ra| \\
& \le \|H_\vp\|_{\cB^2_{\pi/2}\to \cB^2_{\pi/2}}
\sum_{j=1}^4\|g_j\|_{\cB^2_{\pi/2}}\|h_j\|_{\cB^2_{\pi/2}}\\
  & \le
C_1 \|H_\vp\|_{\cB^2_{\pi/2}\to \cB^2_{\pi/2}} \|f\|_{\cB^1_\pi}.
\end{align*}
Therefore,
$
\|\Lambda_f\|_{\cH^*} \le C_1 \|f\|_{\cB^1_\pi}$.  Conversely, let
$\Lambda\in\cH^*$. By Hahn--Banach, we can extend $\Lambda$ to a bounded linear
functional $\widetilde \Lambda$ on the space of
compact operators on $\cB^2_{\pi/2}$, having the same norm.  Hence,
there exists a trace class operator $T:= \sum_{k=1}^\infty \lambda_k\la \cdot
\,|\, u_k\ra v_k$ on $\cB^2_{\pi/2}$ that represents $\widetilde \Lambda$,
where $\{u_k\},\{v_k\}$ are orthonornal bases of $\cB^2_{\pi/2}$, and
$(\lambda_k)\in\ell^1$. 
Notice that $f:=\sum_{k=1}^\infty \lambda_k u_kv_k\in\cB^1_\pi$, and that
\begin{align*}
\Lambda(H_\vp) & = \widetilde \Lambda(H_\vp)
 : = \operatorname{tr} \big(T^*H_\vp\big) =    \sum_{j=1}^\infty \la
 H_\vp u_j
\,|\, Tu_j\ra \\
& = \sum_{k=1}^\infty \lambda_j \la H_\vp u_j \,|\, v_j\ra = \la
                   \vp\,|\, f\ra.
\end{align*}
Hence, $\Lambda=\Lambda_f$ and the norm equivalence follows from the
open mapping theorem. \qed \ms 

 In the previous theorem, we have used an equivalent condition which guarantees that the sequence $\{f_n\}$ weakly
 converges to $0$. For the reader's convenience, we include the proof of this statement, taking $\kappa=\pi$ (instead of $\kappa=\pi/2$).
\begin{lem}\label{da-togliere}
Let $\{f_n\}$ be a sequence in $\cB^2_\pi$ Then, the following two
conditions are equivalent:
\begin{itemize}
\item[(i)]
 $\{f_n\}$ is weakly
 converging to $0$;
\item[(ii)] there exists $C>0$ independent of $n$ such that 
$\|f_n\|_{\cB^2_\pi}\le C$,  and $f_n\to
0$ uniformly on compact subsets.
\end{itemize}
\end{lem}

\proof
$(i) \Rightarrow (ii)$. The uniform boundedness principle gives that
there exists $C>0$ such that 
$\|f_n\|_{\cB^2_\pi}\le C$ for every $n\in\bbN$. For $j\in\bbZ$, let $K_j
: = \sinc(\cdot-j)$.
Then $\{K_j\}$ is an orthonormal basis for $\cB^2_\pi$ (see \eqref{repr-form:eq}).
Fix a compact set $E\subseteq\bbC$ and $\eps>0$.  Then,
there exists $j_E$ such that if $|j|>j_E$ and $z\in E $ we have
$$
|K_j(z)|\le \frac{C}{|j|} \quad\text{and}\quad C^2\sum_{|j|>j_E} \frac{1}{j^2}<\eps.
$$
Then, for $z\in E$
\begin{align*}
|f_n(z)| 
& \le \sum_{|j|\le j_E} |f_n(j) \sinc(z-j)| + \sum_{|j|> j_E} |f_n(j) \sinc(z-j) |\\
& \le \sum_{|j|\le j_E} |f_n(j)| + \|f_n\|_{\cB^2_\pi} \sum_{|j|> j_E}
 \frac{C^2}{j^2} \\
 & \le \sum_{|j|\le j_E} |\la f_n \,|\, \sinc(\cdot-j)\ra| +\eps.
\end{align*}
Now, since $\{f_n\}$  weakly
 converges to $0$, we can select a positive integer $n_\eps$ large enough so that if
$n\ge n_\eps$
$$
\sum_{|j|\le j_E} |\la f_n \,|\, \sinc(\cdot-j)\ra| < \eps.
$$
This shows that $(i) \Rightarrow (ii)$. \ms

Conversely, let $g\in\cB^2_\pi$, $\eps>0$ be given.   Let
$E\subseteq\bbR$ be a compact set such that $\|
\chi_{E^c}g_0\|_{L^2(\bbR)}<\eps$. Then, 
\begin{align*}
  |\la f_n\,|\, g\ra | &
  \le \bigg( \int_E +\int_{E^c} \bigg) |f_n(x)g(x)|\, dx\\
& \le \sup_{x\in E}|f_n(x)| |E|^{1/2} \|g\|_{\cB^2_\pi} + C \|\chi_{E^c}g_0\|_{L^2(\bbR)}
\end{align*}
By selecting $n$ large enough so that $\sup_{x\in E}|f_n(x)|
|E|^{1/2}<\eps$, the conclusion follows.
\qed

\ms

\section{Final remarks}\label{final-sec}

The Bernstein spaces are particular cases of de Branges spaces, see,
for example,~\cite{dB} and~\cite{Romanov}. Indeed the entire function
$e^{-i\pi z}$ is a Hermite Biehler function and, if $1\leq p\leq \infty, $
$$
\cB^p_\pi=\left\lbrace f \text{ entire } :  f(z)/e^{-i\pi z},\  f^{\#}(z)/e^{-i\pi z}  \in H^p(\bbC^+)\right\rbrace=\cH^p(e^{-i\pi\cdot}) .
$$  
 If $E$ is any Hermite Biehler and 
 $\alpha \in [0,1)$,  we observe that, when $z \in \bbC^+$,
$$
\Re\left( \frac{e^{i\pi \alpha} E(z)+e^{-i\pi \alpha}E^\#(z)}{e^{i\pi \alpha} E(z)-e^{-i\pi \alpha}E^\#(z)}\right)> 0 .
$$
Therefore, thanks to the Herglotz theorem~\cite[Theorem 4]{dB}, there exists a positive, Poisson integrable, Borel measure $\sigma_\alpha$ such that
$$
\frac{e^{i\pi \alpha} E(z)+e^{-i\pi \alpha}E^\#(z)}{e^{i\pi \alpha} E(z)-e^{-i\pi \alpha}E^\#(z)}=ip_\alpha z+c_{\alpha}+\int_{\bbR}\left( \frac{1}{t-z}-\frac{t}{1+t^2}\right)d\sigma_\alpha(t).
$$
The measures $\sigma_\alpha$ are the Clark measures associated to
$E$. For a complete description of these measures, we refer
to~\cite[Chapter 9]{Cauchytransform} and to
~\cite{clarkintroduction}. 

If $E(z)=e^{-i\pi z}$, we are able to compute explicitely the  associated Clark measures. Indeed
$$
\supp{\sigma_\alpha}=\left\lbrace  t \in \mathbb{R} : e^{-i\pi \alpha}e^{i\pi t}=e^{i\pi \alpha}e^{-i\pi t}	\right\rbrace= \left\lbrace  t \in \mathbb{R} : e^{2i\pi t}=e^{2i\pi \alpha}	\right\rbrace =\left\lbrace  t=n+\alpha  : n \in \mathbb{Z} \right\rbrace=:\mathbb{Z}_\alpha
$$
and 
$$
\sigma_\alpha(n+\alpha)=\frac{1}{|(e^{-i \pi (\cdot )})'|}=\frac{1}{\pi}.
$$
It follows that the space $\cX_\alpha$ can be equivalently described through the Clark measures. 
\begin{cor}
Let $\alpha\in[0,1)$. Then
$$
    \cX_\alpha  = \Big\{ f\in \cE_\pi:\,    f(iy) =o(|y|e^{\pi|y|})\ 
      \text{for}\ |y|\to\infty, \,  f(\cdot)/e^{-i\pi (\cdot)} \in \BMO(\sigma_\alpha) \Big\}\Big/\operatorname{span}\{e^{\pm
        i\pi(\cdot)}\} ,
      $$
      with norm
      $$
      \|f\|_{\cX_\alpha}=  \pi \| f(\cdot)/e^{-i\pi (\cdot)}\|_{\BMO(\sigma_\alpha)} .      
      $$
\end{cor}
This last description of the space $\cX_\alpha$ is similar to the one presented in ~\cite[Theorem 2]{Bessonov}.

\ms

Some of the results presented in this work, hold also while describing the dual of $\cH^1(E)$ when $E(z)$ is a general Hermite-Biehler function. We will present these theorems in a different article. 
Nevertheless, the Bernstein space $\cB^1_\pi$ satisfies some exceptional properties: because $H^1(\bbZ)$ is isomorphic to $\cB^1_\pi$, we are able to provide an atomic description for $\cB^1_\pi$.
We first recall the definition of the atoms of $H^1(\bbZ)$.
\begin{defn}{\rm 
A sequence  $\alpha=(\alpha_n)$ is an atom of $H^1(\bbZ)$ if it
satisfies these three conditions:
\begin{itemize}
  \tbi
  $\alpha$ has compact (i.e., finite) support, contained in an
  interval $A$;
  \tbi $\sup_n |\alpha_n|\le 1/\#A$, where $\# A$ denotes the cardinality
  of $A$;
  \tbi $\alpha$ has mean value 0, that is, $\sum_{n\in I} \alpha_n=0$. 
\end{itemize}  
}
\end{defn}

\begin{defn}{\rm
The atoms $a$ of the space $\cB^1_\pi$ are the functions
\begin{equation}
\label{atom B1pi}
a(z):= \sum_n (-1)^n \alpha_n \ \sinc(\pi(z-n))\ ,
\end{equation}
where $(\alpha_n)$ is an atom of $H^1(\bbZ)$. }
\end{defn}

If $f \in \cB^1_\pi$, the sequence $(f_n)$, where
$f_n:=(-1)^nf(n)$ belongs to $H^1(\bbZ)$. Consequently, for every
$\eps >0$, there exists a family of atoms $\{\alpha^j\}$, $\alpha^j =
(\alpha^j_n)$ such that
\begin{align*}
\|(f_n) -\sum_j \lambda_j (\alpha^j) \|_{H^1(\bbZ)}< \eps  .
\end{align*}
 From~\cite{Eoff}, we know that $\widetilde
T_0: H^1(\bbZ)\to\cB^1_\pi$ is an isomorphism. It follows that  
\begin{align*}
\Big\| f-\sum_{j} \lambda_j\Big(\sum_{n}(-1)^n\alpha^j_n\ \sinc(\pi(\cdot -n))\Big) \Big\|_{\cB^1_\pi}&\le C \|(f_n) -\sum_j \lambda_j (\alpha^j) \|_{H^1(\bbZ)}< C\eps .
\end{align*}

\ms Another interstenting properties that we will study in the future,
is the atomic decomposition of the de Branges space $\cH^1(E)$. Indeed
it will be interesteng to characterize for which Hermite-Biehler
function, the space $\cH^1(E)$ admits an atomic decomposition. We
think that the one-component Hermite-Biehler functions, used for example in
~\cite{Barbernstein, Bessonov}, will play a fundamental role.

\bibliography{dualB1}
\bibliographystyle{amsalpha}
\vspace{-.1cm}

\end{document}